\documentclass[reqno,11pt]{amsart}
\usepackage{amssymb, amsmath}
\usepackage{graphicx}
\usepackage{cite}
\usepackage{tensor}
\usepackage{appendix}
\usepackage{comment}

\usepackage{subfig, tikz}
\usetikzlibrary{decorations.pathmorphing}
\usepackage{graphics}

\newcommand{\R}{\mathbb{R}}
\newcommand{\C}{\mathbb{C}}
\newcommand{\CP}{\mathbb{CP}}
\newcommand{\Z}{\mathbb{Z}}

\usepackage{bm}
\usepackage{enumerate}
\theoremstyle{plain}
\newtheorem{theorem}{Theorem}[section]
\newtheorem{lemma}[theorem]{Lemma}
\newtheorem{prop}[theorem]{Proposition}
\newtheorem{cor}[theorem]{Corollary}

\theoremstyle{definition}

\newtheorem{definition}[theorem]{Definition}

\newtheorem{example}[theorem]{Example}

\def\be{\begin{equation}}
\def\ee{\end{equation}}
\def\bea{\begin{eqnarray}}
\def\eea{\end{eqnarray}}

\DeclareMathOperator{\arcsinh}{arcsinh}

\newcounter{mnotecount}[section]

\renewcommand{\themnotecount}{\thesection.\arabic{mnotecount}}

\newcommand{\mnote}[1]
{\protect{\stepcounter{mnotecount}}$^{\mbox{\footnotesize
$
\bullet$\themnotecount}}$ \marginpar{
\raggedright\tiny\em
$\!\!\!\!\!\!\,\bullet$\themnotecount: #1} }

\numberwithin{equation}{section}

\begin{document}
\title{Quasi-Einstein structures and Hitchin's equations}
\author{Alex Colling}
\address{Department of Applied Mathematics and Theoretical Physics\\ 
University of Cambridge\\ Wilberforce Road, Cambridge CB3 0WA, UK.}
\email{aec200@cam.ac.uk}
\author{Maciej Dunajski}
\address{Department of Applied Mathematics and Theoretical Physics\\ 
University of Cambridge\\ Wilberforce Road, Cambridge CB3 0WA, UK\\and\\
Faculty of Physics,  University of Warsaw\\
 Pasteura 5, 02-093 Warsaw, Poland}
\email{m.dunajski@damtp.cam.ac.uk}

\date{May 8, 2026}

\begin{abstract}
We prove (Theorem 1.1.) that a class of 
quasi-Einstein structures on closed manifolds must admit a Killing vector field. This extends the rigidity theorem
obtained in \cite{DL23} for the extremal black hole horizons and completes the classification of compact quasi-Einstein 2-manifolds in this class. We also explore special cases of the quasi-Einstein equations related to integrability and the Hitchin equations, as well as to Einstein-Weyl structures and Kazdan-Warner type PDEs. This leads to novel explicit examples of quasi-Einstein structures on (non-compact) 2-manifolds and on $S^2 \times S^1$.
\end{abstract}

\maketitle

\section{Introduction}
A quasi-Einstein manifold is an $n$-dimensional manifold $M$ together with a Riemannian metric $g$ and a vector field $X \in \mathfrak{X}(M)$ satisfying the system of equations
\be
\label{QEE}
\mbox{Ric}(g)=\frac{1}{m}X^\flat\otimes X^\flat
-\frac{1}{2}{\mathcal L}_{X} g+\lambda g.
\ee
Here $m$ and $\lambda$ are constants with $m \neq 0$,  $\mathcal{L}_X$ denotes the Lie derivative along $X$ and $X^\flat$ is the 1-form $g$-dual to $X$. We refer to (\ref{QEE}) as the quasi-Einstein equations (QEE). They may be viewed as a generalisation of the Einstein equations, which they reduce to when the vector field $X$ vanishes identically. We will be interested in quasi-Einstein manifolds $(M,g,X)$ with non-zero $X$ and call a solution  trivial if $X \equiv 0$.

The QEE appear in different geometric contexts for various special values of $m$. The most notable case is $m = 2$, for which the QEE describe near-horizon geometries of extremal black hole spacetimes in vacuum with cosmological constant $\lambda$ \cite{KL13,LP03}. Another special value is $m = 2-n$, for which any solution to the QEE defines an Einstein-Weyl structure \cite{C43}. For $m = 1-n$ with $\lambda = 0$ there is a relation to projective geometry: the QEE describe projective structures $[\nabla]$ that contain both a Levi-Citiva connection and a representative connection with skew-symmetric Ricci tensor \cite{NR16,CDKL24}. For gradient quasi-Einstein structures, i.e. when $X^\flat$ is a gradient, additional cases of interest are $m \in \Z_{> 0}$ because of the correspondence between the QEE and Einstein metrics on the warped product of $M$ with an $m$-dimensional manifold \cite{KK03}. Finally, we may formally include the Ricci soliton equation, which does not contain the term $\frac 1m X^\flat \otimes X^\flat$, as the case $m = \infty$.

In many applications it is natural to assume the manifold $M$ is compact. For $m = 2$ it is shown in \cite{DL23} that any non-gradient quasi-Einstein structure on a closed manifold admits a Killing vector field. Our main result in Section \ref{qeerig} is a generalisation of the arguments in \cite{DL23} to all values of $m$ outside a finite interval.

\begin{theorem}  \label{RQEi}
Let $(M,g)$ be a closed Riemannian $n$-manifold admitting a non-gradient vector field $X$ such that the QEE (\eqref{QEE}) hold with either (i) $m \geq 2$ or (ii) $m \leq 2-n$. Then $(M,g)$ admits a Killing vector field $K$. Moreover, $[K,X] = 0$.
\end{theorem}

For $n = 2$ the existence of a Killing vector follows for $m \geq 2$ or $m < 0$ (recall $m \neq 0$). Together with Corollary \ref{topQEE} below and the work \cite{CDKL24}, this completes the classification of compact quasi-Einstein 2-manifolds with $m \notin (0,2)$. 

\begin{cor} \label{surfrig}
    Any closed orientable two-dimensional quasi-Einstein manifold with $m \notin (0,2)$ is either trivial or given by the family of solutions on $S^2$ presented in \textup{\cite{CDKL24}}.
\end{cor}

In any dimension, if $m \leq 2-n$ and $\lambda \leq 0$ the QEE imply that the Ricci tensor is ``too negative-definite" to admit a Killing vector as in Theorem~\ref{RQEi}. This allows us to prove

\begin{theorem} \label{negmthm}
    Any closed quasi-Einstein manifold with $\lambda \leq 0$ and $m \leq 2-n$ is trivial.
\end{theorem}

In Section \ref{closedrig} we consider the implications of Theorem \ref{RQEi} for solutions for which $X^\flat$ is closed but not exact. When $m > 0$ this assumption on $X^\flat$ implies that $X$ is parallel and $M$ is locally a product of an Einstein manifold and $S^1$ \cite{BGKW22, KL24,W23}. We show in Theorem~\ref{statnegm} that for $m < 0$ solutions can be constructed more generally as \textit{warped} products of (gradient) \textit{quasi}-Einstein manifolds with $S^1$, and moreover when $m \leq 2-n$ any solution is locally of this form.

The proof of Theorem \ref{RQEi} for $m = 2-n$ can be deduced from an argument by Tod \cite{T92} showing the existence of a Killing vector for compact Einstein-Weyl structures. We review this argument in Section \ref{EW} and show how it relates to the QEE. Given an Einstein-Weyl structure, the QEE are equivalent to an elliptic PDE of the type studied originally by Kazdan and Warner \cite{KW74}. This leads us to some explicit examples of the quasi-Einstein structures discussed in Section \ref{closedrig} on $S^2 \times S^1$ that to the best of our knowledge are new.

In two dimensions the QEE impose topological restrictions on a closed 2-manifold $M$, which we turn to in Section \ref{top}. We consider a generalisation of the two-dimensional quasi-Einstein equations where we allow $\lambda$ to be an arbitrary function. This equation was introduced on a 2-manifold in \cite{KL24}. It appears for example for $m = 2$ in the context of extremal horizons in Einstein-Maxwell theory, where the matter content effectively promotes the constant $\lambda$ to a function \cite{LP03}.

\begin{definition}(\!\!\cite[Definition 3]{KL24}). \label{GEH}
    A metric $g$ and vector field $X$ on a manifold $M$ satisfy the \textit{generalised extremal horizon equations} for some function $\Lambda \in C^\infty(M)$ and constant $c \neq 0$ if 
    \begin{equation} \label{gehe}
        \mathcal{L}_X g + c\:  X^\flat \otimes X^\flat + \Lambda g = 0.
    \end{equation}
\end{definition}
The value of $c$ can be set arbitrarily at the price of rescaling $X$ and $\Lambda$. Note that in two dimensions the Ricci tensor is proportional to the metric, which implies that the QEE on a 2-manifold are a special case of (\ref{gehe}). In higher dimensions a particular case of interest occurs when the metric $g$ is Einstein: the generalised extremal horizon equations then describe Einstein-Weyl structures for which some metric in the associated conformal class is Einstein.

For the two-dimensional case, it is shown by Lewandowski and Kaminski \cite[Proposition 4]{KL24} (see also the earlier work  \cite{DRSL} applicable to horizons)
 that any solution to (\ref{gehe}) on a closed, connected and oriented 2-manifold for which $X \not\equiv 0$ must be on the two-sphere. We will present an alternative proof based on the Poincaré-Hopf theorem, from which we deduce the following:
\begin{theorem} \label{toprig}
    Let $(g,X)$ be a solution to the generalised extremal horizon equations on a closed and connected $n$-dimensional manifold $M$ with $X$ not identically zero.
    \begin{enumerate}[(i)]
        \item Suppose $n = 2$. Then  $\chi(M) > 0$.
        \item Suppose $n > 2$. If $n$ is even and $\Lambda$ and $X$ do not vanish simultaneously, then $\chi(M) > 0$. If $n$ is odd, $\Lambda$ must vanish somewhere.
    \end{enumerate}
\end{theorem}
Setting $\Lambda = R - 2\lambda$ and $c = -\frac 2m$ to obtain the quasi-Einstein equations on a 2-manifold, we deduce
\begin{cor}\textup{(\!\!\cite[Corollary 1]{KL24}).} \label{topQEE}
    Any non-trivial quasi-Einstein structure on a closed, connected and oriented 2-manifold is on a 2-sphere.
\end{cor}
In Section \ref{hitchin} we consider the QEE on a 2-manifold with $m = -1$ and $\lambda = 0$, which are the parameters related to projective differential geometry in two dimensions \cite{derdzinski, BDE, krynski, R14, WY16}. Using a flat connection obtained in \cite{CDKL24} we construct a Lax pair for the QEE and show that they arise as a symmetry reduction of the anti-self-dual Yang-Mills equations. In this way we can identify the QEE as a special case of the Hitchin equations \cite{H87} with gauge group $SU(2)$ on the complexified tangent bundle $E=TM\otimes\C$,
\be
\label{h11n}
    F = -[\Phi,\Phi^*],  \quad
    \overline{\mathcal{D}}_A \Phi = 0.
\ee
Here $F \in \Omega^{1,1}(M ; \text{End }E)$ is the curvature of a connection $A$ on $E$ and $\Phi$ is a traceless section of $\Omega^{1,0}(M; \text{End }E)$ which, as a consequence of (\ref{h11n}), is holomorphic with respect to the holomorphic structure induced by the metric $g$ and $A$. $\Phi^*$ is the adjoint of $\Phi$ and the bracket $[\cdot, \cdot]$ denotes the usual extension of the Lie bracket to matrix-valued forms. In the following we also make use of the decomposition
\begin{equation*}
    E  = T^{1,0}M \oplus T^{0,1}M
\end{equation*}
induced by $g$ to define projection operators $\mathcal{P}^{1,0}: E \to T^{1,0} M$ and $\mathcal{P}^{0,1}: E \to T^{0,1} M$, as well as the $(1,0)$- and $(0,1)$-parts of $X$ and $X^\flat$.
\begin{theorem}
\label{theoremHitchin}
The quasi-Einstein equations for $(g, X)$ on a 2-manifold $M$ with $m = -1$ and $\lambda = 0$ are equivalent to the $SU(2)$ Hitchin equations on $E$ for $(A,\Phi)$ defined by
\footnote{ \label{fn1} An obvious advantage of a coordinate--free notation we employed  is its coordinate invariance. One of the disadvantages is obscurity. $\Phi$ and the second term in (\ref{t1}) are End($E$)-valued 1-forms, i.e. sections of $(T^* M)^2\otimes TM \otimes \C$. By convention the first $T^*M$ factor refers to the 1-form and the second one to the endomorphism. In Section \ref{sub_hitch} we make it concrete:
If $z$ is a local complex coordinate on $M$ such that $g=2e^{H}\text{d}z\text{d}\bar{z}$ and
$X^{\flat}=P\text{d}z+\bar{P}\text{d}\bar{z}$, then $\mathcal{P}^{1, 0}= \text{d}z\otimes \partial_z$
and $(X^{\flat})^{(0, 1)}=\bar{P} \text{d}\bar{z}$ and 
\[
\Phi=\bar{P}\;\text{d}z\otimes \partial_z\otimes
\text{d}\bar{z}=\begin{bmatrix} 0  & \bar{P} \\ 0 & 0 \end{bmatrix}\text{d}z
\]
as the basis of $\mathfrak{gl}(2)$ corresponding to endomorphisms is
\[
\partial_z\otimes \text{d} z= \begin{bmatrix} 1  & 0 \\ 0 & 0 \end{bmatrix}, 
\quad
\partial_z\otimes \text{d} \bar{z}= \begin{bmatrix} 0  & 1 \\ 0 & 0 \end{bmatrix}, \quad
\partial_{\bar{z}}\otimes \text{d} z= \begin{bmatrix} 0  & 0 \\ 1 & 0 \end{bmatrix},\quad
\partial_{\bar{z}}\otimes \text{d} \bar{z}= \begin{bmatrix} 0  & 0 \\ 0 & 1 \end{bmatrix}.
\]
   }
\begin{subequations}
\begin{align}
    \mathcal{D}_A &= \nabla -\frac 12 \left((X^\flat)^{(1,0)}-(X^\flat)^{(0,1)}\right) \otimes (\mathcal{P}^{1,0} - \mathcal{P}^{0,1}),  \label{t1}\\
    \Phi &=  \mathcal{P}^{1,0} \otimes (X^\flat)^{(0,1)} \label{t2}.
\end{align}
\end{subequations}
If (\ref{h11n}) hold, then the Levi--Civita connection $\nabla$ of $g$ is projectively equivalent to a connection with totally anti--symmetric Ricci tensor.
\end{theorem}
We describe this identification in detail and deduce that on any open set where $X$ is not zero, the rescaled metric $h = \vert X \vert^2 g$ has constant negative curvature. It follows  that the Ansatz (\ref{t1}), (\ref{t2}) is gauge equivalent to the canonical 
solution to the $SU(2)$ Hitchin equations on the upper half--plane $M=\{(x, y)\in \R^2, y>0\}$.
Finding the relevant
gauge transformation is equivalent to solving a PDE
\begin{equation} \label{alphaeq}
    y^2(\partial_x^2\Theta + \partial_y^2\Theta) + y(\sin \Theta\: \partial_y \Theta + \cos \Theta\: \partial_x \Theta) + \cos \Theta = 0, \quad\mbox{where}\quad \Theta:M\rightarrow \R.
\end{equation}

 To show that the compactness assumption in Theorem \ref{RQEi} is essential, in the final Section \ref{swni} we construct non-gradient quasi-Einstein 2-manifolds admitting a homothety, but no Killing vectors. There is a 2-parameter family of such solutions for any $m$ with $\lambda = 0$, but in the case $m = -1$ they can be obtained explicitly in terms of hypergeometric functions using the fact that the metric $h$ has constant curvature. This section and the proof of Theorem \ref{toprig} make use of equations obtained by applying the prolongation procedure to the QEE (see \cite{NR16, CDKL24}) and the generalised extremal horizon equations. We review these equations in Appendix~\ref{Aa} and use them to prove that any homothety must preserve $X^\flat$ unless $g$ is flat in Appendix~\ref{Ais}. Appendix~\ref{Ab} contains an exposition of the hypergeometric functions involved in the explicit solutions admitting a homothety.

\subsection*{Acknowledgements} We would like to thank Nigel Hitchin, Wojciech Kami\'nski, James Lucietti, and Paul Tod  for useful discussions. AC acknowledges the support of the Cambridge International Scholarship.  We are grateful to the Isaac Newton Institute for Mathematical Sciences, Cambridge, for support and hospitality during
 the programme Twistor Theory, when some of the results in this paper have been obtained.

\section{Rigidity of quasi-Einstein manifolds} \label{qeerig}

In \cite{DL23} it is shown that any closed quasi-Einstein manifold $(M,g,X)$ with $m = 2$ and $X$ non-gradient must admit a Killing vector field. The proof in \cite{DL23} uses a remarkable tensor identity satisfied by solutions to the $m = 2$ QEE. The derivation of the identity relies on certain cancellations that occur only for $m = 2$. An attempt at generalising the argument to other $m$ has been made in \cite{C24}. In this section we show that, by employing a different identity, the existence of a Killing vector still follows for all values of $m$ outside a finite interval. This completes the proof of Theorem~\ref{RQEi} for $m \in (-\infty,2-n)\cup(2,\infty)$. The case $m = 2-n$ is covered in Section \ref{EW}.

\subsection{Tensor identity} \label{tensoridsec} Following \cite{C24}, we make an Ansatz for the Killing vector $K$ of the form
\begin{equation} \label{Kansatz}
    K = \frac 2m \Gamma X + \nabla \Gamma,
\end{equation}
where $\Gamma$ is a (at this point arbitrary) smooth positive function. The assumption that $X$ is non-gradient ensures that $K$ is non-zero. Solving for $X$, we have
\begin{equation} \label{Xansatz}
    X = \frac{m}{2\Gamma}(K - \nabla \Gamma).
\end{equation}
We will make use of abstract index notation to present the required tensor identity. In this notation, the QEE read
\begin{equation*} 
    R_{ab} = \frac 1mX_a X_b - \nabla_{(a}X_{b)} + \lambda g_{ab}.
\end{equation*}
They can be written in terms of $K$ and $\Gamma$ using (\ref{Xansatz}) as
\begin{align} \label{QEEK}
    R_{ab} &= \frac{m}{4\Gamma^2}(K_a - \nabla_a \Gamma)(K_b - \nabla_b \Gamma) - \frac m2\nabla_{(a}(\Gamma^{-1}K_{b)}) + \frac m2\nabla_{(a}(\Gamma^{-1}\nabla_{b)}\Gamma) + \lambda g_{ab}\nonumber\\[5pt]
    &=\frac{m}{4\Gamma^2} K_aK_b -\frac{m}{2\Gamma}\nabla_{(a}K_{b)} + \frac{m}{2\Gamma}\nabla_a\nabla_b\Gamma - \frac{m}{4\Gamma^2}(\nabla_a\Gamma)(\nabla_b\Gamma)  + \lambda g_{ab}.
\end{align}
We begin with the following intermediate result, which is valid for any $m$. We denote the $g$-norm on tensors by $\vert \cdot \vert$ and write $\Delta = \nabla^a\nabla_a$ for the Laplacian.

\begin{lemma} \label{lemid}
    Any solution to \eqref{QEE} with $X$ given by \eqref{Xansatz} satisfies
    \begin{align}   \label{idstep1}
        \nabla^a\nabla_{(a}K_{b)} &-\frac{1}{\Gamma}(\nabla^a\Gamma)\nabla_{(a}K_{b)}= \nonumber\\&\frac{1}{2\Gamma}K_b\nabla_cK^c+\frac{1}{2\Gamma}K^a\nabla_aK_b-\frac{1}{\Gamma^2}K_bK^c\nabla_c\Gamma 
        -\frac{1}{2\Gamma}(\Delta \Gamma)\nabla_b\Gamma-\frac{3}{2\Gamma}(\nabla^a \Gamma)\nabla_a\nabla_b \Gamma \nonumber\\[5pt]&+\frac{1}{\Gamma^2}\vert \nabla \Gamma \vert^2\nabla_b\Gamma+\nabla^a\nabla_b\nabla_a\Gamma 
        -\frac 12\Gamma\nabla_b\left(\frac{1}{2\Gamma^2}\vert K \vert^2 - \frac{1}{2\Gamma^2}\vert \nabla \Gamma \vert^2 - \frac{1}{\Gamma}\nabla_cK^c + \frac{1}{\Gamma}\Delta \Gamma\right). 
    \end{align}
\end{lemma}
Equation (\ref{idstep1}) is derived by substituting the expression (\ref{QEEK}) for the Ricci tensor into the contracted Bianchi identity $\nabla^a(R_{ab}-\frac 12Rg_{ab}) = 0$.
We present  the full proof in  Appendix \ref{Ac}.

The next step is to use the Ricci identity
\begin{equation*}
    \Delta \nabla_b \Gamma = \nabla_b \Delta \Gamma + R_{ab}\nabla^a\Gamma
\end{equation*}
and (\ref{QEEK}) again to evaluate the triple derivative term $\nabla^a\nabla_b\nabla_a\Gamma$ in (\ref{idstep1}). Compared to the calculation in \cite{DL23}, this step introduces various additional terms that vanish precisely when $m =2$. After contracting with $K^b$ and multiplying by $\Gamma^{\frac m2-1}$, almost all terms can be combined into divergences and terms proportional to $\nabla_a(\Gamma^{\frac m2-1}K^a)$ in the following way:

\begin{prop} \label{propid}
    For any solution to \eqref{QEE} with $X$ given by \eqref{Xansatz} and $m \neq 2$ the following identity holds
    \begin{align} \label{identity}
         &\Gamma^{\frac m2-1}\nabla_{(a}K_{b)}\nabla^aK^b + \frac{1}{m-2}\Gamma^{\frac m2 - 1}(\nabla_aK^a)^2 = \nonumber\\
        &\nabla_a\left(\Gamma^{\frac m2-1}K_b\nabla^{(a}K^{b)} -\frac 14(m-2)\vert \nabla \Gamma \vert^2\Gamma^{\frac m2 - 2}K^a - \frac 12\Gamma^{\frac m2-1}(\nabla_cK^c) K^a - \frac 12\Gamma^{\frac m2 -1}(\Delta \Gamma) K^a -\lambda \Gamma^{\frac m2}K^a\right) \nonumber\\
        &\hspace{2cm}+\nabla_a\left(\Gamma^{\frac m2-1}K^a\right)\left(-\frac{\vert K \vert^2}{2\Gamma} + \frac 12\Delta\Gamma + \frac 14(m-2)\frac{\vert \nabla \Gamma \vert^2}{\Gamma} + \frac{m}{2(m-2)}\nabla_cK^c + \lambda \Gamma\right).
    \end{align}
\end{prop}
A detailed proof of this Proposition is presented in Appendix \ref{Ac}.

\noindent We are now in a position to present the proof of the main Theorem \ref{RQEi}.
\begin{proof}[Proof of Theorem \ref{RQEi}] It is shown in \cite{DL23,T92} that there exists a (unique up to a multiplicative constant) smooth positive function $\Psi$ solving the elliptic PDE
\begin{equation} \label{pdeex}
    \Delta \Psi + \nabla_a (\Psi X^a) = 0.
\end{equation}
Note that this result holds for any vector field $X$ on a compact Riemannian manifold $(M,g)$ and in particular does not use the QEE. We now fix $\Gamma$ to be $\Psi^{\frac 2m}$. Substituting in $\Psi = \Gamma^{\frac m2}$ and the expression (\ref{Xansatz}) for $X$, equation (\ref{pdeex}) is equivalent to
\begin{equation*}
    \nabla_a(\Gamma^{\frac m2-1}K^a) = 0.
\end{equation*}
Hence, for this choice of $\Gamma$ the vector field $\Gamma^{\frac m2-1}K$ with $K$ defined by (\ref{Kansatz}) is divergence-free. If $(M,g,X)$ satisfies the QEE, $K$ satisfies the identity (\ref{identity}) by Proposition \ref{propid}. Observe that the RHS of this identity reduces to a divergence for our choice of $\Gamma$. Integrating the identity over the closed manifold $M$ using the divergence theorem yields
\begin{equation} \label{intid}
    \int_M \Gamma^{\frac m2-1}\left(\nabla_{(a}K_{b)}\nabla^aK^b + \frac{1}{m-2}(\nabla_a K^a)^2\right)\text{ vol}_g = 0.
\end{equation}
For $m > 2$ both terms are non-negative and this immediately implies $\nabla_{(a}K_{b)} = 0$, i.e. $K$ is a Killing vector. For the case $m < 2-n$, note that any $(0,2)$ tensor $T_{ab}$ satisfies 
$$T_{ab}T^{ab} \geq \frac 1n(g^{ab}T_{ab})^2$$ at every point. Applying this inequality with $T_{ab} = \nabla_{(a}K_{b)}$ shows 
\begin{equation*}
    \nabla^aK^b\nabla_{(a}K_{b)} \geq \frac 1n(\nabla_a K^a)^2,
\end{equation*}
so the integrand of (\ref{intid}) is bounded below by a positive multiple of $\Gamma^{\frac m2-1}(\nabla_a K^a)^2$ if $m < 2-n$. It follows that $\nabla_a K^a = 0$ and therefore also $\nabla_{(a}K_{b)} = 0$. For $m = 2-n$ this argument only shows that $K$ is a conformal Killing vector and we need to refer to Theorem~\ref{ewthm}. 

To prove that $K$ preserves $X$, recall that we fixed $\Gamma$ such that 
\begin{equation*}
    0 = \Gamma^{2-\frac m2}\nabla_a(\Gamma^{\frac m2-1}K^a) = \Gamma \nabla_a K^a + \frac 12(m-2)K^a\nabla_a \Gamma.
\end{equation*}
If $K$ is Killing and $m \neq 2$ we obtain $K^a\nabla_a \Gamma = 0$, which implies $[K,X] = 0$.
\end{proof} 

\subsection{Rigidity results} The most general solution to the QEE on a 2-manifold admitting a Killing vector is presented in \cite{CDKL24}. In this work it is also investigated which of these solutions extend to a two-sphere. Combined with Theorem \ref{RQEi} and the results in Section \ref{top}, we obtain a complete classification of quasi-Einstein structures on a compact 2-manifold with $m \notin (0,2)$. 

\begin{proof}[Proof of Corollary \ref{surfrig}]
It is shown in \cite[Theorem 1.2]{CSW11} that any closed quasi-Einstein 2-manifold with $X$ a gradient must be trivial\footnote{In \cite{CSW11} only positive values of $m$ were considered, but the proof of \cite[Theorem 1.2]{CSW11} can be generalised to any $m \neq 0$.}. We may therefore assume that $X$ is non-gradient, in which case the existence of a Killing vector is guaranteed for $m \notin (0,2)$ by Theorem \ref{RQEi}. Moreover, Corollary~\ref{topQEE} implies that any such solution must be on the two-sphere. The result now follows from the classification \cite[Theorem~1.1]{CDKL24}.
\end{proof}

Returning to the $n$-dimensional case, once we know that $K$ is Killing and $[K,X]=0$, it is useful to revisit the identity in Lemma \ref{lemid}. The arguments below are analogous to those in \cite{KL24} for $m = 2$.
\begin{cor} \label{corconst}
    Let $(M,g,X)$ be a closed and connected quasi-Einstein manifold such that either $X$ is a gradient or $m$ is outside the interval $(2-n,2)$. Then the function
    \begin{equation} \label{const}
        A = -\frac{\vert K \vert^2}{2\Gamma} + \frac 12\Delta \Gamma + \lambda\Gamma  + \frac{m-2}{4\Gamma}\vert \nabla \Gamma \vert^2
    \end{equation}
    is constant everywhere on $M$.
\end{cor}
\begin{proof}
Since $K$ is a Killing vector and $\mathcal{L}_K \Gamma = 0$, the identity (\ref{idstep1}) simplifies significantly and becomes
\begin{align*}
    0 = \:-&\frac{1}{4\Gamma}\nabla(\vert K \vert^2)- \frac{1}{2\Gamma}(\Delta \Gamma)\nabla \Gamma - \frac{3}{4\Gamma}\nabla (\vert \nabla \Gamma \vert^2) + \frac{1}{\Gamma^2}\vert \nabla \Gamma\vert^2\nabla\Gamma \nonumber\\
    &+ \Delta \nabla\Gamma - \frac 12\Gamma\nabla\left(\frac{1}{2\Gamma^2}\vert K \vert^2 - \frac{1}{2\Gamma^2}\vert \nabla \Gamma \vert^2 + \frac{1}{\Gamma}\Delta \Gamma\right).
\end{align*}
This may be further reduced to
    \begin{equation} \label{geq1}
        \nabla \left( -\frac{1}{2\Gamma}\vert K \vert^2 - \frac 12\Delta \Gamma - \frac{1}{2\Gamma}\vert \nabla\Gamma \vert^2\right) + \Delta\nabla \Gamma = 0.
        \end{equation}
Note that when $X^\flat = \nabla f$ is a gradient this equation holds even without the compactness assumption for any $m$ with $\Gamma = e^{-\frac 2m f}$ and $K \equiv 0$.
    Applying the Ricci identity and using (\ref{QEEK}), we obtain
    \begin{equation} \label{geq2}
        \Delta \nabla \Gamma = \nabla \Delta \Gamma + \text{Ric}(\nabla \Gamma) = \nabla \left(\Delta \Gamma + \frac{m}{4\Gamma}\vert \nabla \Gamma \vert^2 + \lambda \Gamma \right).
    \end{equation}
    From (\ref{geq1}) and (\ref{geq2}) we find d$A = 0$ with $A$ as in (\ref{const}).
\end{proof}
For $m = 2-n$ the expression (\ref{const}) reduces to the Gauduchon constant \cite{G95} written in the quasi-Einstein frame (see Section \ref{EW}). For $m = 2$ it is related to a symmetry enhancement for near-horizon geometries \cite{DL23, KLR07}. In the gradient case the first term on the right-hand side of (\ref{const}) vanishes and $A$ reduces to the constant that appears in the relation with warped product Einstein metrics \cite[Equation 1.3]{KK03}. By applying a maximum principle to the remaining terms in the expression for $A$, it is shown in \cite{KK03} that for $\lambda \leq 0$ the function $\Gamma$ must be constant. 
\begin{prop} \textup{(\!\!\cite{KK03})}\label{grad} 
    Any closed quasi-Einstein manifold with $\lambda \leq 0$ and $X$ a gradient is trivial.
\end{prop}
\noindent When $m \leq 2-n$ we can deduce this result even in the non-gradient case using Theorem \ref{RQEi}.

\begin{proof}[Proof of Theorem \ref{negmthm}]
By Proposition \ref{grad} we may assume that $X$ is non-gradient. Theorem \ref{RQEi} then implies that there is a Killing vector $K$ such that $[K,X] = 0$. Any Killing vector satisfies
    \begin{equation*} \label{Bochner}
        \frac 12 \Delta \vert K \vert^2 = \vert \nabla K \vert^2 - \text{Ric}(K,K).
    \end{equation*}
We contract the QEE twice with the Killing vector $K$ and use the fact that $K$ and $X$ commute to write $(\mathcal{L}_Xg)(K,K) = \mathcal{L}_X(\vert K \vert^2)$. We find
\begin{equation*} 
    \text{Ric}(K,K) = \frac 1m g(K,X)^2 - \frac 12\mathcal{L}_X(\vert K \vert^2) + \lambda \vert K \vert^2.
\end{equation*}
Combining these equations leads to
\begin{equation} \label{maxK}
    \Delta (\vert K \vert^2) - \mathcal{L}_X(\vert K \vert^2)  = 2\vert \nabla K \vert^2 - \frac 2m g(K,X)^2 - 2\lambda \vert K \vert^2.
\end{equation}
Since the RHS is non-negative, an application of the maximum principle \cite[Theorem 2.9]{K93} shows that either $K \equiv 0$ or $\lambda = 0$ and $\vert K \vert^2$ is constant. In the latter case (\ref{maxK}) shows that $K$ is parallel and $g(K,X) = 0$. However, since $K$ is of the form (\ref{Kansatz}) with $\mathcal{L}_K \Gamma = 0$, we have $g(K,X) = \frac m2\Gamma^{-1} \vert K \vert^2$. We conclude that $K \equiv 0$, contradicting the fact that $X$ is non-gradient.
\end{proof}
We expect Theorem \ref{negmthm} to hold for all $m < 0$. Extending the above proof would require a version of Theorem \ref{RQEi} for all negative $m$. Note that for $n = 2$ the result can also be deduced from Corollary~\ref{topQEE} by integrating the trace of the QEE using the Gauss-Bonnet theorem as in \cite[Proposition 2.1]{CDKL24}.

\subsection{Quasi-Einstein structures with closed $X^\flat$ and $m < 0$} \label{closedrig}
 Quasi-Einstein structures $(g,X)$ on a closed manifold $M$ with $m > 0$ for which $X^\flat$ is closed but not exact have been completely classified \cite{BGKW22, KL24,W23}. This condition is of particular interest for $m = 2$: the corresponding near-horizon geometries are static and exhibit a local AdS$_3$ symmetry \cite{KLR07}.

\begin{theorem} \textup{(\!\!\cite{BGKW22, KL24,W23})} \label{statposm}
    Let $(M,g,X)$ be a closed quasi-Einstein manifold with $m > 0$ and $X^\flat$ closed but not exact. Then $\lambda < 0$ and $(M,g,X)$ is a quotient of $(\widetilde{M},\widetilde{g},\widetilde{X})$ by a discrete group of isometries, where
    \begin{equation*}
        \widetilde{M} = N \times \R, \hspace{.5cm} \widetilde{g} = g_N \oplus \textup{d}t^2, \hspace{.5cm} \widetilde{X} = \sqrt{-m\lambda} \: \frac{\partial}{\partial t}.
    \end{equation*}
Here $t$ is a coordinate on $\R$ and $(N,g_N)$ is a simply connected negative Einstein manifold with $\textup{Ric}(g_N) = \lambda g_N$. 
\end{theorem}

By a quotient of a quasi-Einstein manifold in the theorem above we mean a solution satisfying $\widetilde{X} = \pi^* X$ and $\widetilde{g} = \pi^* g$, where $\pi: \widetilde{M} \to M$ is the quotient map.

Solutions with $m < 0$ turn out to be less rigid. In particular, the vector field $X$ need not be parallel, as the explicit examples in Section \ref{ewex} show. Using Theorem~\ref{RQEi} we can prove for $m \leq 2-n$ that the general solution is not a direct product of $\R$ with an Einstein manifold, but rather a warped product of $\R$ with a (gradient) quasi-Einstein manifold.

\begin{theorem} \label{statnegm}
    Let $(M,g,X)$ be a closed quasi-Einstein manifold of dimension $n$ with $m \leq 2-n$ and $X^\flat$ closed but not exact. Then $\lambda > 0$ and $(M,g,X)$ is a quotient of $(\widetilde{M},\widetilde{g},\widetilde{X})$ by a discrete group of isometries, where
    \begin{equation} \label{splitmneg}
        \widetilde{M} = N \times \R, \hspace{.5cm} \widetilde{g} = g_N \oplus e^{-\frac 2m u}\textup{d}t^2, \hspace{.5cm} \widetilde{X}^\flat = \textup{d}t + \textup{d}u.
    \end{equation}
    Here $t$ is a coordinate on $\R$, $u$ is a function on $N$ and $(N,g_N)$ is a simply connected quasi-Einstein manifold with vector field given by $X_N^\flat  = \frac{m+1}{m}\textup{d}u$ and parameters $(m+1,\lambda)$.
\end{theorem}
\begin{proof}
    By Theorem \ref{negmthm} we may assume $\lambda > 0$. Theorem \ref{RQEi} implies that there is a function $\Gamma > 0$ such that $K$ defined by (\ref{Kansatz}) is a Killing vector satisfying $\mathcal{L}_K \Gamma = 0$. Hence $K$ is also a Killing vector of the rescaled metric $h = \Gamma^{-1}g$. Moreover, if $X^\flat$ is closed, so is the 1-form $\alpha = \Gamma^{-1}K^\flat$ which is $h$-dual to $K$. We find that $K$ is parallel with respect to $h$, which implies that the universal cover $\widetilde{M}$ of $M$ splits isometrically as a product $N \times \R$. Let us write $(\widetilde{g},\widetilde{X}, \widetilde{\Gamma})$ for the pullback of $(g,X,\Gamma)$ to $\widetilde{M}$. We choose a coordinate $t$ on $\R$ such that the pullback of the parallel form $\alpha$ is $\widetilde{\alpha} = \frac 2m\text{d}t$ and introduce a function $u$ by $\widetilde{\Gamma} = e^{-\frac 2mu}$. Then 
    \begin{equation*}
        \widetilde{X}^\flat = \frac m2(\widetilde{\alpha} - \widetilde{\Gamma}^{-1}\text{d}\widetilde{\Gamma}) = \text{d}t + \text{d}u.
\end{equation*}
After shifting $u$ by a constant if necessary (recall that $K$ and $\Gamma$ are only defined up to a multiplicative constant) we have $\tilde{g}_{tt} = e^{-\frac 2mu}$ and the universal cover $(\widetilde{M},\widetilde{g},\widetilde{X})$ is of the form (\ref{splitmneg}).

We now impose the QEE (\ref{QEE}) on the Ansatz (\ref{splitmneg}). Let us denote the pullback of the Ricci tensor Ric$(\tilde{g})$ to $N$ via the inclusion $\iota: N \to N \times \R$ by Ric$(\tilde{g})\vert_N$. Then we have
\begin{subequations}
    \begin{align}
        \text{Ric}(\tilde{g})\vert_N &= \text{Ric}(g_N) - \frac{1}{m^2} \nabla u \otimes \nabla u + \frac 1m\text{Hess}_N(u), \\
        \text{Ric}(\tilde{g})(\partial_t,\partial_t) &= \frac 1me^{-\frac 2mu}\left(\Delta_N u - \frac 1m\vert \nabla u \vert_N^2\right).  \label{rictt}
    \end{align}
\end{subequations}
    The QEE impose 
\begin{subequations}
    \begin{align}
        \text{Ric}(\tilde{g})\vert_N &= \frac 1m\nabla u \otimes \nabla u - \text{Hess}_N(u) + \lambda g_N, \\
        \text{Ric}(\tilde{g})(\partial_t,\partial_t) &= \frac 1m + \frac 1me^{-\frac 2mu}\vert \nabla u \vert_N^2 + \lambda e^{-\frac 2mu}.  \label{rictt2}
    \end{align}
\end{subequations}
    The cross terms $\text{Ric}(\tilde{g})(Y,\partial_t)$ with $Y \in \Gamma(TN)$ vanish. Equating the components tangent to $N$ gives 
    \begin{equation*}
        \text{Ric}(g_N) = \frac{m+1}{m^2}\nabla u \otimes \nabla u - \frac{m+1}{m}\text{Hess}_N(u) + \lambda g_N.
    \end{equation*}
    which are the QEE\footnote{Note that for $m = -1$ the manifold $(N,g_N)$ is Einstein. In this case the QEE for (\ref{splitmneg}) are equivalent to the PDE (\ref{uconst}) on the Einstein manifold $(N,g_N)$ (see Section \ref{EW}).} for $g_N$ with $X^\flat_N = \frac{m+1}{m}\text{d}u$ and parameters $(m+1,\lambda)$. 
    \end{proof}
Observe that the $tt$-component of the QEE for $(\widetilde{M},\widetilde{g},\widetilde{X})$ does not impose any further restrictions on the quasi-Einstein manifold $(N,g_N,X_N)$ if $m \neq -1$. Indeed, equating (\ref{rictt}) and (\ref{rictt2}) yields
    \begin{equation} \label{uconst}
        \Delta_N u = m\lambda + e^{\frac 2m u} + \frac{m+1}{m}\vert \nabla u \vert_N^2.
    \end{equation}
    Writing $\Gamma_N = Ce^{-\frac 2mu}$ for some constant $C > 0$, (\ref{uconst}) is equivalent to
    \begin{equation*}
        -\frac 1m C = \frac 12 \Delta_N \Gamma_N + \lambda \Gamma_N + \frac{m-1}{4\Gamma_N}\vert \nabla \Gamma_N \vert_N^2.
    \end{equation*}
    This is the statement that the function $A$ defined in~(\ref{const}) equals $-\frac 1mC$. An application of the maximum principle shows that $A$ is positive for (closed) gradient quasi-Einstein manifolds with $\lambda > 0$, so we may always set $C = -mA$. Therefore, if $(N,g_N,X_N)$ solves the QEE with parameters $(m+1,\lambda)$ and $X_N = -\frac m2 \Gamma_N^{-1}\nabla\Gamma_N$, then the Ansatz (\ref{splitmneg}) solves the QEE with parameters $(m,\lambda)$, provided $u$ is defined by
    \begin{equation*}
    \Gamma_N = -mAe^{-\frac 2mu}.
    \end{equation*}

If $u$ is constant (i.e. the quasi-Einstein structure on $N$ is trivial) we recover the $m < 0$ analogue of the solutions in Theorem \ref{statposm}. As $\lambda$ is positive for $m < 0$, Lemma~\ref{grad} does not apply and further examples with for which $u$ is not constant can be constructed as follows. For $m > 0$ and $\lambda > 0$ non-trivial compact gradient quasi-Einstein manifolds are known to exist from the correspondence with warped product Einstein metrics (see \cite{B98,LPP04}). Using the duality between gradient quasi-Einstein structures with $m = \mu$ and $m = 2-n-\mu$ \cite{C12} these yield solutions with $m < 2-n$. Solutions with $2-n < m < 0$ can also be obtained using the warped product construction in \cite[Proposition 5.14]{C12}.

The correspondence between quasi-Einstein structures with $m = \mu$ and $m = 2-n-\mu$ implies that the manifolds $N$ in Theorem \ref{statnegm} admit complete quasi-Einstein metrics with $m > 0$ and $\lambda > 0$. By a quasi-Einstein version of Myers's theorem \cite{Q97} (see \cite{L10} for the non-gradient case), $N$ must be compact. Since compact two-dimensional gradient quasi-Einstein manifolds are trivial, we conclude
\begin{cor}
    Let $(M,g,X)$ be a closed quasi-Einstein manifold of dimension $3$ with $m < -1$ and $X^\flat$ closed but not exact. Then $\lambda > 0$ and $(M,g,X)$ is a quotient of $(\widetilde{M},\widetilde{g},\widetilde{X})$ by a discrete group of isometries, where
    \begin{equation*}
        \widetilde{M} = S^2 \times \R, \hspace{.5cm} \widetilde{g} = (\lambda^{-1} g_{S^2}) \oplus \textup{d}t^2, \hspace{.5cm} \widetilde{X} = \sqrt{-m\lambda} \: \frac{\partial}{\partial t}.
    \end{equation*}
Here $t$ is a coordinate on $\R$ and $g_{S^2}$ is the round metric on $S^2$.
\end{cor}
\noindent For $m = -1$ every solution is locally equivalent to one of the examples presented in Section \ref{ewex}.

\section{Einstein--Weyl structures} \label{EW}
On a manifold of dimension $n \geq 3$ the QEE with $m = 2-n$ can be identified with the Einstein-Weyl equation in a special frame. In this section we use this correspondence to classify closed quasi-Einstein manifolds with $m = 2-n$ and construct new solutions on $M = S^2 \times S^1$.

\begin{definition}
A \textit{Weyl structure} on a manifold $M$ is a conformal structure $[g]$ with a torsion-free connection $D$ satisfying 
\begin{equation} \label{ws}
    D g = \omega \otimes   g
\end{equation}
for some $g \in [g]$ and $\omega \in \Omega^1(M)$. A Weyl structure $([g],D)$ satisfies the \textit{Einstein-Weyl condition} if the symmetric part of the Ricci tensor Ric$(D)$ is proportional to $[g]$.
\end{definition}
We may view a Weyl structure as being defined by the pair $(g,\omega)$, since the condition (\ref{ws}) then determines $D$. If we pick a different representative $e^ug$  of $[g]$ for some $u \in C^\infty(M)$, the 1-form $\omega$ satisfying (\ref{ws}) changes to $\omega + \text{d}u$. Hence $(e^ug,\omega + \text{d}u)$ corresponds to the same Weyl structure as $(g,\omega)$. Choose a frame $(g,\omega)$ and define a vector field $X$ by $X^\flat = \frac{n-2}{2}\omega$. The Einstein-Weyl condition then reads \cite{PT93}
\begin{equation} \label{ewg}
    \text{Ric}(g) = \frac{1}{2-n}X^\flat \otimes X^\flat - \frac 12\mathcal{L}_X g + \Lambda g.
\end{equation}
Here $\Lambda = \Lambda[g,\omega]$ is a function on $M$ that depends on the frame in the following way. 

\begin{lemma} \label{lemconf}
    Under a change of frame to $(\tilde{g},\widetilde{\omega})= (e^ug, \omega + \textup{d}u)$, the function $\Lambda$ transforms as
    \begin{equation} \label{lambdaconf}
        \Lambda[\tilde{g},\widetilde{\omega}] = \frac 12e^{-u}(-\Delta u + \mathcal{L}_X u+2\Lambda[g,\omega]).
    \end{equation}
\end{lemma}
\begin{proof}
Write $\widetilde{\nabla}$ for the Levi-Civita connection of $\tilde{g}$ and let $\widetilde{X}$ be the vector field $\tilde{g}$-dual to $\frac{n-2}{2}\widetilde{\omega}$. For any 1-form $\alpha$, we have 
\begin{align*}
    \widetilde{\nabla} \alpha = \nabla \alpha - \frac 12 \alpha \otimes \nabla u - \frac 12 \nabla u \otimes \alpha + \frac 12\langle \alpha^\sharp, \nabla u\rangle \: g. 
\end{align*}
All quantities on the RHS are calculated with respect to $g$ and $\langle \cdot , \cdot\rangle$ denotes the pairing between vectors and covectors. Using this we calculate
\begin{align*}
    \text{Hess}_g(u) &= \text{Hess}_{\tilde{g}}(u) + \nabla u \otimes \nabla u - \frac 12\vert \nabla u \vert^2_g\:g, \\
    \mathcal{L}_X g &= \mathcal{L}_{\widetilde{X}}\tilde{g} - (n-2)\text{Hess}_{\tilde{g}}(u) + \frac{n-2}{2}(\nabla u \otimes \widetilde{\omega} + \widetilde{\omega} \otimes \nabla u) - (n-2)\nabla u \otimes \nabla u - \langle X,\nabla u\rangle g.
\end{align*}
The Ricci tensor transforms as
\begin{equation*}
    \text{Ric}(\tilde{g}) = \text{Ric}(g) + \frac{2-n}{4}\left(2\:\text{Hess}_g(u) - \nabla u \otimes \nabla u\right) - \frac 14\left(2\Delta_g u +(n-2)\vert \nabla u \vert^2_g\right)g.
\end{equation*}
Using (\ref{ewg}) with $X^\flat = \frac{n-2}{2}(\widetilde{\omega} - \text{d}u)$, we find that the Einstein-Weyl condition in the new frame reads
\begin{equation*}
    \text{Ric}(\tilde{g}) = \frac{1}{2-n}\left(\tfrac{n-2}{2}\widetilde{\omega}\right)\otimes\left(\tfrac{n-2}{2}\widetilde{\omega}\right) - \frac 12\mathcal{L}_{\widetilde{X}}\widetilde{g} + \frac 12e^{-u}(-\Delta_g u + \langle X, \nabla u \rangle + 2\Lambda[g,\omega])\tilde{g}.
\end{equation*}
From here we read off (\ref{lambdaconf}).
\end{proof}

\subsection{Quasi-Einstein frames} Equation (\ref{ewg}) shows the QEE with $m = 2-n$ describe Einstein-Weyl structures in a \textit{quasi-Einstein frame} where the function $\Lambda$ is a constant $\lambda$. To find such frames on a compact manifold, it is useful to first consider the equations in another frame constructed in \cite{T92}.
\begin{theorem} \textup{(\!\!\cite{T92})}\label{Todthm}
    For any Weyl structure on a closed manifold $M$ there exists a frame $(g,\omega)$ (unique up to homothety) such that the vector field $\omega^\sharp$ is divergence-free with respect to $g$. If the Weyl structure satisfies the Einstein-Weyl condition, $\omega^\sharp$ is a Killing vector of $g$.
\end{theorem}
We refer to this frame as the \textit{Gauduchon frame}. The proof of Theorem \ref{Todthm} is very similar to the arguments in Section \ref{tensoridsec}. It relies on the same existence result for the PDE (\ref{pdeex}) to find a frame in which $\omega^\sharp$ is divergence-free and then uses the contracted Bianchi identity and the Einstein-Weyl condition to express $\vert \mathcal{L}_{\omega^\sharp}g \vert^2$ as a total divergence in this frame.

Writing $K$ for the Killing vector in the Gauduchon frame defined by $K^\flat = \frac{n-2}{2}\omega$, the Einstein-Weyl condition in this frame reduces to
\begin{equation} \label{gauew}
    \text{Ric}(g) = \frac{1}{2-n}K^\flat \otimes K^\flat + \Lambda g.
\end{equation}
Note that $\mathcal{L}_K \Lambda = 0$ since $K$ is a Killing vector. Starting from the Gauduchon frame, let us look for a function $u \in C^\infty(M)$ such that $(e^ug, \omega + \text{d}u)$ is in a quasi-Einstein frame with constant $\lambda$. Using Lemma \ref{lemconf}, we find that $u$ must satisfy
\begin{equation} \label{KW}
    \Delta u = 2\Lambda + \mathcal{L}_K u - 2\lambda e^u.
\end{equation}
This Kazdan-Warner type PDE has been studied in the context of the Chern-Yamabe problem \cite{ACS17}. In particular, it is shown in \cite{Y24} that (\ref{KW}) has a unique solution on a compact manifold if $\lambda < 0$ and the integral of $\Lambda$ is negative. It will follow from Theorem \ref{negmthm} that $\lambda$ and the integral of $\Lambda$ need to be positive in order to find (non-trivial) quasi-Einstein frames\footnote{In fact, using a similar argument as in the proof of Theorem \ref{negmthm} applied to (\ref{gauew}) it can be shown that $\Lambda$ is strictly positive pointwise in dimension $n \geq 4$ if $K \not\equiv 0$.}. Solving (\ref{KW}) is more subtle in this positive curvature case. If a solution exists, we have the following.
\begin{lemma} \label{uinv}
    Let $(M,g)$ be a closed Riemannian manifold with Killing vector $K$. If $u \in C^\infty(M)$ satisfies (\ref{KW}) for some constant $\lambda$ and function $\Lambda$ with $\mathcal{L}_K \Lambda = 0$, then $\mathcal{L}_K u = 0$.
\end{lemma}
\begin{proof}
We multiply (\ref{KW}) by $\mathcal{L}_K u$ and use the fact that $K$ is Killing to write
\begin{equation*}
    (\Delta u)\mathcal{L}_K u =  \text{div}((\mathcal{L}_Ku)(\nabla u)) - g(\nabla u, \mathcal{L}_K(\nabla u)) = \text{div}((\mathcal{L}_Ku)(\nabla u)) - \frac 12\mathcal{L}_K(\vert \nabla u \vert^2).
\end{equation*}
Using $\mathcal{L}_K \Lambda = 0$, the resulting equation reads
\begin{equation*}
    (\mathcal{L}_K u)^2 = \text{div}\left((\mathcal{L}_K u )(\nabla u)\right) + \mathcal{L}_K\left(2\lambda e^u - 2\Lambda u - \tfrac 12\vert \nabla u \vert^2\right).
\end{equation*}
We now integrate this identity over the compact manifold $M$ using the divergence theorem. Since $K$ is divergence-free, the RHS integrates to zero. As the integrand on the LHS is non-negative, we must have $\mathcal{L}_K u = 0$. 
\end{proof}
\noindent We summarise our results in the theorem below.

\begin{theorem} \label{ewthm}
Any quasi-Einstein structure $(g,X)$ on a closed manifold $M$ with $m = 2-n$ is of the form
\begin{equation} \label{qeframe}
    g = e^u g_0, \hspace{.8cm} g(X, \cdot) = g_0(K, \cdot) +\frac{n-2}{2}\textup{d}u,
\end{equation}
with $u \in C^\infty(M),\: K \in \mathfrak{X}(M)$ and $g_0$ a Riemannian metric. The triple $(g_0,K,u)$ satisfies $\mathcal{L}_K g_0 = 0$ and $\mathcal{L}_K u = 0$, together with 
\begin{subequations}
\begin{align} 
    \textup{Ric}(g_0) &= \frac{1}{2-n}K^\flat \otimes K^\flat + \Lambda g_0,\label{gaudEW}\\
    \Delta_0 u &= 2\Lambda - 2\lambda e^u. \label{pdeu}
\end{align}
\end{subequations}
Here $\Lambda \in C^\infty(M)$ and $K^\flat$ denotes the 1-form $g_0$-dual to $K$.
\end{theorem}
\begin{proof}
    Any quasi-Einstein structure $(g,X)$ satisfies the Einstein-Weyl condition (\ref{ewg}), and therefore the pair $(g,\omega)$ with $g(X,\cdot) = \frac{n-2}{2}\omega$ defines an Einstein-Weyl structure. By Theorem \ref{Todthm} there exists a Gauduchon frame $(g_0,\omega_0)$  in which the vector field $K$ $g_0$-dual to $\frac{n-2}{2}\omega_0$ is Killing and the Einstein-Weyl condition reduces to (\ref{gaudEW}). If we define the function $u$ by $g = e^ug_0$, the data $(g,X)$ is related to $(g_0,K)$ by (\ref{qeframe}). Moreover, $u$ satisfies equation~(\ref{KW}), which reduces to (\ref{pdeu}) since $\mathcal{L}_K u = 0$ by Lemma \ref{uinv}.
\end{proof}
In particular, $K$ is also a Killing vector of the quasi-Einstein metric $g$. Up to a multiplicative constant, $K$ is of the form (\ref{Kansatz}) with $\Gamma = e^{u}$ and $m = 2-n$. It also satisfies $[K,X] = 0$, which completes the proof of Theorem \ref{RQEi} for this value of $m$. Note that Theorem \ref{negmthm} implies that any (non-trivial) solution has $\lambda > 0$.
\subsection{Examples} \label{ewex}
    If $X$ is a gradient, we necessarily have $K \equiv 0$. Equation (\ref{gaudEW}) then implies that $g_0$ is $\Lambda$-Einstein (in particular, $\Lambda$ is constant). Hence any gradient solution to the $m = 2-n$ QEE is conformal to Einstein and obtained by solving (\ref{pdeu}) on a compact $\Lambda$-Einstein manifold. This result can be viewed as a special case of \cite[Corollary 4.15]{C12}. By the maximum principle, solutions necessarily have $\lambda$ and $\Lambda$ positive. As an example where (\ref{pdeu}) is explicitly solvable, we may take 
    \begin{equation*}
        M = N \times S^2, \hspace{.8cm} g = g_N \oplus g_{S^2}.
    \end{equation*}Here $g_{S^2}$ is the round metric on $S^2$ and $(N,g_N)$ is an Einstein manifold satisfying Ric$(g_N) = g_N$. Assuming $u$ is a function on $S^2$, we have $\Lambda = 1$ and $u$ satisfies 
    \begin{equation}\label{constcurv}
        \Delta_{S^2}u = 2 - 2\lambda e^u.
    \end{equation}
    Equation (\ref{constcurv}) is the condition for the conformally rescaled metric $e^ug_{S^2}$ to have constant scalar curvature $2\lambda$. Since any two metrics of the same constant curvature on $S^2$ are isometric, $\lambda e^u$ must be the conformal factor associated to a conformal diffeomorphism of the sphere. Recall that the round metric on $S^2$ can be written in terms of a complex coordinate $z$ (coming from stereographic projection) as
    \begin{equation*}
        g_{S^2} = \frac{4\text{d}z\text{d}\bar{z}}{(1 + \vert z\vert^2)^2}.
    \end{equation*}
    Conformal diffeomorphisms are given by M\"obius transformations $f(z) = \frac{az + b}{cz+d}$ with $ad - bc \neq 0$. This leads to the general solution to (\ref{constcurv}),
    \begin{equation} \label{mobius}
        e^u = \frac{1}{\lambda}\left(\frac{1 + \vert z \vert^2}{1 + \vert f(z) \vert^2}\left\vert \frac{\text{d}f}{\text{d}z}\right\vert\right)^2.
    \end{equation}
    The corresponding quasi-Einstein structures are given by (\ref{qeframe}). \\
    
For $n = 3$ equation (\ref{gaudEW}) can be solved completely whether or not $K$ vanishes \cite{T92}. The solutions include (i) constant curvature spaces, (ii) a solution on $S^2 \times S^1$, (iii) the Berger sphere and (iv) a 4-parameter family containing a 3-parameter subfamily that is globally defined on $S^3$. To solve the QEE it remains to solve (\ref{pdeu}) on these backgrounds. For the Einstein-Weyl structure in case (ii) on $M = S^2 \times S^1$ this can be done explicitly. It is given by
    \begin{equation*} 
        g_0 = g_{S^2} \oplus \text{d}t^2, \hspace{.8cm} K = \frac{\partial}{\partial t},
    \end{equation*}
    where $t$ is a periodic coordinate on $S^1$ and $g_{S^2}$ is the round metric on $S^2$. The function $\Lambda$ is a positive constant, which we have set to $1$ by rescaling $g_0$ and $K$. The function $u$ satisfies $\mathcal{L}_K u = \partial_t u = 0$. Hence, we may view $u$ as a function on $S^2$ solving (\ref{constcurv}), so that $u$ is given by (\ref{mobius}). The quasi-Einstein metric is $g = e^ug_0$, and the 1-form $g$-dual to $X$ is $X^\flat = \text{d}t + \frac 12\text{d}u$. It follows from Theorem~\ref{statnegm} that this is the general solution to the QEE with $n = 3$ and $m = -1$ for which $X^\flat$ is closed but not exact.

\section{Topology of quasi-Einstein 2-manifolds} \label{top}

In this section we consider the topology of compact quasi-Einstein 2-manifolds and, more generally, of compact solutions to the generalised extremal horizon equations (\ref{gehe}) introduced in \cite{KL24}. We present a proof of Theorem \ref{toprig} that uses the Poincaré-Hopf theorem and is based on a careful analysis of the zeros of a vector field $X$ solving the generalised extremal horizon equations.

\begin{theorem}\textup{(Poincaré-Hopf)}
    Let $X$ be a vector field on a closed manifold $M$ having only isolated zeros. Then the Euler characteristic $\chi(M)$ equals the sum of the indices of the zeros of $X$.
\end{theorem}
Recall that the index of $X$ at an isolated zero $p$ may be defined as follows: let $D$ be a closed ball around $p$ (of dimension $n = \text{dim } M$) contained in a coordinate chart such that $p$ is the only zero of $X$ in $D$. Using the chart to view $X$ as a map $D \to \R^n$, the index at $p$ is the topological degree of the map $X/\vert X \vert: \partial D \to S^{n-1}$.

\subsection{Zeros of $X$} 
Our proof is inspired by arguments in \cite{CST18} and uses a result by Milnor to calculate the indices of the zeros of $X$.

\begin{lemma}\textup{(\!\!\cite[Lemma 6.4]{M65})} \label{mindex} Let $p \in M$ be a zero of a vector field $X$ on a manifold $M$. If in local coordinates \textup{det}$(\partial_\mu X^\nu) \neq 0$ at $p$, then the zero is isolated and the index equals the sign of $\textup{det}(\partial_\mu X^\nu)$.
\end{lemma}
In case the determinant vanishes, further information is needed to compute the index. The following lemma provides criteria that are sufficient to conclude that the index is non-negative.

\begin{lemma} \label{mindex2}
    Let $p \in M$ be an isolated zero of a vector field $X$ on a manifold $M$. If there exists a chart $U \ni p$ such that \textup{det}$(\partial_\mu X^\nu) \geq 0$ on $U$, the index of $X$ at $p$ is non-negative. Moreover, if \textup{det}$(\partial_\mu X^\nu)$ is strictly positive on $U \setminus \{p\}$, the index is positive.
\end{lemma}
\begin{proof} The idea is to deform $X$ slightly as in the proof of the Poincaré-Hopf theorem with degenerate zeros \cite{M65} to obtain a non-degenerate vector field $\widetilde{X}$ to which we can apply Lemma \ref{mindex}. By shrinking $U$, we can assume that $p$ is the only zero of $X$ in $U$. Set $n = \text{dim } M$ and view $X$ as a map $U \to \R^n$ using the induced trivialisation of $TM$. Let $p \in V \subset\subset W \subset\subset U$ be compactly contained open balls around $p$ and suppose $y \in \R^n$ is a regular value of $X$. We define
\begin{equation}
    \widetilde{X} = X - y\chi,
\end{equation}
where $\chi$ is a bump function that is equal to $1$ in $V$ and vanishes outside $W$. Since regular values are dense by Sard's theorem, we can choose $y$ close enough to $0$ such that $\widetilde{X}$ does not vanish in $W \setminus V$. The index of $X$ at $p$ is equal to the sum of the indices of the zeros of $\widetilde{X}$ within $V$. Indeed, both can be evaluated as the degree of a map $\partial W \to S^{n-1}$, where $X$ and $\widetilde{X}$ agree. Since $\chi$ is constant in $V$, $y$ is a regular value and det$(\partial_\mu X^\nu) \geq 0$, we must have det$(\partial_\mu \tilde{X}^\nu) > 0$ at any zero of $\widetilde{X}$ in $V$. Hence using Lemma \ref{mindex} we find that the index of $X$ at $p$ equals the number of pre-images of $y$ in $V$, which is non-negative. The final claim follows from the fact that if \textup{det}$(\partial_\mu X^\nu)$ is strictly positive on $U \setminus \{p\}$, we can choose $y$ to be in the image in $X$, so that there is at least one pre-image.
\end{proof}

\subsection{Two-dimensional case}\label{2dim} We will make use of some equations obtained by applying the prolongation procedure to the two-dimensional generalised extremal horizon equations. For the quasi-Einstein case, these equations were derived in \cite{CDKL24, NR16} (see also Appendix \ref{Aa}). Let us introduce a function $\Omega$ such that d$X^\flat = \Omega \epsilon$, where $\epsilon$ is a (local) volume form of $g$. The prolonged generalised extremal horizon equations read
\begin{equation} \label{prolong}
    2 \nabla X^\flat =- cX^\flat \otimes X^\flat  -\Lambda g + \Omega\epsilon.
\end{equation}
We will need the following differential consequences of (\ref{prolong}), where $R$ denotes the scalar curvature of $g$ and $\star$ is the Hodge star operator. The proof can be found in Appendix \ref{Aa}.
\begin{lemma} \label{prolonglem}
    Any solution to the generalised extremal horizon equations on a 2-manifold satisfies 
\begin{align} \label{domega}
    \textup{d}\Omega &= -\frac 32 c\Omega X^\flat + (\tfrac 12 c\Lambda - R)\star\! X^\flat +\star \textup{d} \Lambda, \\
    \Delta \Lambda &= -\frac 32 c\Omega^2 - \langle X, \textup{d}(2c\Lambda - R)\rangle + (\tfrac 12 c\Lambda - R)(\Lambda - c\vert X \vert^2). \label{laplam}
\end{align}
\end{lemma}
\noindent We are now in a position to present the proof of Theorem \ref{toprig}(i).
\begin{proof}[Proof of Theorem \ref{toprig}(i)]
We first observe that $X$ must have at least one zero. This is well-known for the quasi-Einstein and near-horizon Einstein-Maxwell equations \cite{J09,BGKW23, CKL24} and follows in the same way for the generalised extremal horizon equations (see Section \ref{sechd} for the proof). Hence, by the Poincar\'e-Hopf theorem it suffices to prove that $X$ has only isolated zeros of positive index.  To show any zero must be isolated, introduce a local complex coordinate $z$ (with respect to the complex structure defined by $g$) such that 
\begin{equation} \label{coords}
    g = 2e^H\text{d}z\text{d}\bar{z} \hspace{.8cm}\text{and}\hspace{.8cm} X^\flat = P\text{d}z + \bar{P}\text{d}\bar{z}
\end{equation}
for some real function $H$ and complex function $P$. The $(\bar{z}\bar{z})$ component of the generalised extremal horizon equations reads
\begin{equation} \label{zbzbcomp}
    2\partial_{\bar{z}}\bar{P} - 2\bar{P}\partial_{\bar{z}}H + c\bar{P}^2 = 0.
\end{equation}
There locally exists a complex function $F$ such that $\partial_{\bar{z}}F = \bar{P}$. Using this we can write (\ref{zbzbcomp}) as
\begin{equation} \label{holo}
    \partial_{\bar{z}}\left(e^{\frac 12 cF}e^{-H}\bar{P}\right) = 0.
\end{equation}
It follows that the function $e^{\frac 12cF-H}\bar{P}$ is holomorphic. In particular, every zero of $\bar{P}$ (and hence of $X$) must be isolated and not all derivatives of the components of $X$ vanish at a zero. 

It remains to show the index at any zero $p$ is positive. In any coordinates $(x^1,x^2)$ around a zero $p = (0,0)$ of $X$ we have from (\ref{prolong})
\begin{equation*}
    \text{det}(\partial_\mu X^\nu)\big\vert_p = \frac 14\left(\Lambda(p)^2 + \Omega(p)^2\right). 
\end{equation*}
Using Lemma \ref{mindex}, we find that the zero at $p$ is of index $1$ unless $X, \Lambda$ and $\Omega$ vanish simultaneously. In the degenerate case where $\Lambda(p) = \Omega(p) = 0$, we show that \textup{det}$(\partial_\mu X^\nu)$ has a strict minimum at $p$. Recall that (\ref{holo}) implies there exists a smallest integer $k$ such that not all $(k+1)$-th order partial derivatives of components of $X$ vanish at $p$. From (\ref{prolong}) we find $k \geq 1$ and
\begin{equation*}
\text{det}(\partial_\mu X^\nu) = \frac{1}{4(k!)^2}\left[\left(\Lambda_{,\iota_1\dots\iota_k}(p)x^{\iota_1}\dots x^{\iota_k}\right)^2 + \left(\Omega_{,\iota_1\dots\iota_k}(p)x^{\iota_1}\dots x^{\iota_k}\right)^2\right] + O(\vert x \vert^{2k+1}).
\end{equation*}
Here $\Lambda_{,\iota_1\dots\iota_k}$ denotes a $k$-th order partial derivative of $\Lambda$. To evaluate these derivatives we use the equations in Lemma~\ref{prolonglem}. Evaluating (\ref{domega}) and (\ref{laplam}) to leading order at $p$ yields
\begin{equation} \label{constr}
    \Omega_{,\iota_1\dots\iota_k} = \tensor{\epsilon}{_{\iota_1}^\rho} \Lambda_{,\rho\iota_2\dots\iota_k} \hspace{.4cm}\text{ and (if $k \geq 2$) } \hspace{.4cm}g^{\mu\nu}\Lambda_{,
    \mu\nu\iota_1\dots\iota_{k-2}} = 0\hspace{.2cm}\text{ at $p$.}
\end{equation}
These constraints imply 
\begin{equation} \label{C}
    C = \min_{v \in S^1}\left\{\left(\Lambda_{,\iota_1\dots\iota_k}(p)v^{\iota_1}\dots v^{\iota_k}\right)^2 + \left(\Omega_{,\iota_1\dots\iota_k}(p)v^{\iota_1}\dots v^{\iota_k}\right)^2\right\} > 0.
\end{equation}
Indeed, if there exists $v \in S^1$ such that both terms in the argument of (\ref{C}) are zero, then together with (\ref{constr}) we find that $\Lambda$ and $\Omega$ vanish to order $k$. But then from (\ref{prolong}) $X$ vanishes to order $k+1$, contradicting the definition of $k$. Hence
\begin{equation*}
    \text{det}(\partial_\mu X^\nu) \geq 
            \frac{1}{4(k!)^2}C\vert x \vert^{2k} + O(\vert x\vert^{2k+1})
\end{equation*}
has a strict minimum at $p$. By Lemma \ref{mindex2}, this shows the index at $p$ is also positive in the degenerate case and hence completes the proof. \end{proof}

If $M$ is orientable, since $\chi(M) > 0$ it follows that $M$ is diffeomorphic to $S^2$. Using the complex structure on $M$ induced by $g$ and the fact that any $(0,1)$-form on $S^2$ is $\bar{\partial}$-exact, we can then define the function $F$ introduced in equation (\ref{holo}) globally by $\bar{\partial} F = (X^\flat)^{(0,1)}$. The computation (\ref{holo}) now shows that $e^{\frac 12cF}X^{(1,0)}$ is a globally defined holomorphic vector field. In \cite{KL24} the existence of a global holomorphic vector field is shown without a priori assuming that $M$ is a 2-sphere, which leads to an alternative proof of Theorem \ref{toprig}(i).
    
Since $\chi(S^2) = 2$ the proof of Theorem \ref{toprig}(i) also shows that $X$ is either identically zero or vanishes at two points, which coincide in the degenerate case.
\begin{cor}
For any solution $(g,X)$ to the generalised extremal horizon equations on $S^2$ the vector field $X$ is either identically zero or vanishes at at most two points.
\end{cor}
We note that it is possible for $X$ to have only one zero. To see this, fix any metric $g$ on $S^2$ and let $V$ be a holomorphic vector field with exactly one zero (e.g. $V = z^2\partial_z$ with $z$ a complex coordinate compatible with $g$ coming from stereographic projection). Let $F$ be a globally defined nowhere zero complex-valued function satisfying $\bar{\partial}F = V^\flat$. We define the real vector field $X$ by $(X^\flat)^{(0,1)} = 2(cF)^{-1}V^\flat$. It is now straightforward to verify that the ``trace-free" part of the generalised extremal horizon equations (equation (\ref{holo}) and its complex conjugate) is satisfied, and we can always fix $\Lambda$ such that the remaining trace component holds.

\subsection{Higher dimensions} \label{sechd}
In dimension $n > 2$ the QEE are no longer a special case of the generalised extremal horizon equations. The equations may however still be of interest: for example, it follows from (\ref{ewg}) that any Einstein-Weyl structure $(M,[g],D)$ for which $(M,[g])$ is conformal to Einstein defines a solution to the generalised extremal horizon equations. The arguments in Section \ref{2dim} can be generalised to deduce Theorem \ref{toprig}(ii).
\begin{proof}[Proof of Theorem \ref{toprig}(ii)]
Contracting the generalised extremal horizon equations twice with $X$ yields
\begin{equation*}
    \langle X, \text{d}(\vert X \vert^2)\rangle + c\vert X \vert^4 + \Lambda\vert X \vert^2 = 0.
\end{equation*}
We can eliminate $\Lambda = -\frac 2n\text{div }X - \frac cn\vert X \vert^2$ using the trace of the generalised extremal horizon equations. The resulting equation reads
\begin{equation*}
    \langle X, \text{d}(\vert X \vert^2)\rangle + \frac{c(n-1)}{n}\vert X \vert^4 - \frac 2n\vert X \vert^2\:\text{div }X = 0.
\end{equation*}
This can be written as
    \begin{equation*}
        \text{div}\left(\frac{X}{\vert X \vert^ n}\right) = \frac 12c(n-1)\vert X \vert^{2-n}.
    \end{equation*}
    Since the RHS does not change sign, on compact $M$ this shows that $X$ must have zero. To evaluate the index at a zero, we again consider the prolonged generalised extremal horizon equations
    \begin{equation*} 
        2\nabla X^\flat + cX^\flat \otimes X^\flat  = -\Lambda g + \Omega,
    \end{equation*}
    where $\Omega = \text{d}X^\flat$. Since $\Omega$ is antisymmetric the determinant of the RHS is positive if $\Lambda \neq 0$ and $n$ is even, so any zero is isolated and of index 1 by Lemma \ref{mindex}. As before, we deduce from the Poincaré-Hopf theorem that $\chi(M) > 0$. For odd $n$ the sign of the determinant equals the sign of~$\Lambda$. As closed odd-dimensional manifolds have Euler characteristic equal to zero, we conclude that~$\Lambda$ must change sign and hence vanish somewhere.
    \end{proof}

\section{Hitchin's equations and projective metrizability} \label{hitchin}
For $m = 1-n$ and $\lambda = 0$ the $n$-dimensional quasi-Einstein equations describe projective structures that are both skew and metrizable (see \cite{CDKL24,R14} for details). Concretely, suppose $\nabla$ is the Levi-Civita connection of some metric $g$ and let $\Upsilon \in \Omega^1(M)$. Then the quasi-Einstein equations with $X^\flat = (1-n)\Upsilon$ are exactly the condition for the projectively equivalent connection 
\begin{equation}
    \hat{\nabla} = \nabla +  \Upsilon \otimes \text{Id} + \text{Id}\otimes \Upsilon
\end{equation}
to have a Ricci tensor that is anti--symmetric. Moreover, for $m = 1-n,\: \lambda = 0, \: p = -\frac 12, \: q = \frac{1}{n-1}$ the quasi-Einstein equations are equivalent to the affine connection\footnote{Explicitly, for $Y,Z \in \Gamma(TM)$ we have $$ D_Y Z = \nabla_Y Z - p X^\flat(Y)Z - qX^\flat(Z)Y. $$}
\begin{equation} \label{D}
    D = \nabla - p \: X^\flat \otimes \text{Id} - q\: \text{Id}\otimes X^\flat 
\end{equation}
having vanishing Ricci curvature (note that $D$ does have torsion).

\subsection{Lax pair for  $(m=-1,\lambda=0)$ quasi--Einstein equations} From now on we set $n = 2$ and consider the quasi-Einstein equations on a 2-manifold with $m = -1, \lambda = 0$. In terms of a local complex coordinate $z$ as in (\ref{coords}), the QEE can be written as the system
\begin{subequations} \label{qeez}
\begin{align} 
    \partial_z P - P\partial_z H + P^2 &= 0, \label{zz} \\
    \partial_z\partial_{\bar{z}}H -\frac 12(\partial_z\bar{P} + \partial_{\bar{z}}P) - \vert P \vert^2 &= 0,\label{zzb} \\
    \partial_{\bar{z}}\bar{P}-\bar{P}\partial_{\bar{z}}H + \bar{P}^2 &= 0. \label{zbzb}
\end{align}
\end{subequations}
In these coordinates the connection (\ref{D}) is of the form $D = \text{d} + U$, where (using notation as in footnote \ref{fn1})
\begin{equation*}
    U = \begin{bmatrix}
        \partial_z H -\frac 12 P & - \bar{P} \\ 0 & \frac 12 P 
    \end{bmatrix}\text{d}z + \begin{bmatrix} \frac 12\bar{P} & 0 \\ -P & \partial_{\bar{z}} H - \frac 12\bar{P}\end{bmatrix}\text{d}\bar{z}.
\end{equation*}
We perform a gauge transformation  $U \to \gamma U \gamma^{-1}-\text{d}\gamma \cdot \gamma^{-1} $ with $\gamma = e^{\frac H2}\text{Id}$, which changes the trivialisation of $TM$ to one in which $g$ becomes the standard inner product. The resulting connection 1-form $V$ is traceless and given by
\begin{equation} \label{Vflat}
    V = \begin{bmatrix}
        \frac 12(\partial_z H - P) & - \bar{P} \\ 0 & \frac 12(P -\partial_z H)
    \end{bmatrix}\text{d}z + \begin{bmatrix} \frac 12(\bar{P} - \partial_{\bar{z}}H) & 0 \\ -P & \frac 12(\partial_{\bar{z}} H - \bar{P})\end{bmatrix}\text{d}\bar{z}.
\end{equation}
Note that the QEE hold if and only if $D$ is flat, i.e. if and only if the operators $M = \partial_z + V_z$ and $L = \partial_{\bar{z}} + V_{\bar{z}}$ commute. Moreover, the transformation
\begin{equation*}
    z \mapsto \epsilon z, \hspace{.2cm} \bar{z}\mapsto \epsilon^{-1}\bar{z},\hspace{.2cm} P \mapsto \epsilon^{-1}P, \hspace{.2cm} \bar{P}\mapsto \epsilon \bar{P}, \hspace{.2cm} H \mapsto H
\end{equation*}
preserves $(g,X)$ and hence also the QEE. Setting $\xi = -\epsilon^{-2}$, these observations allow us to introduce a Lax pair with spectral parameter $\xi$.
\begin{prop} 
\label{propasd}
The $(m=-1, \lambda=0)$ quasi--Einstein equations on a 2-manifold admit a Lax formulation. If $(g, X)$ are given by (\ref{coords}), then the QEE are equivalent to
\[
[L, M]=0, \quad {\forall}\;\xi
\]
where 
\begin{align}
    L &= \partial_{\bar{z}} + \begin{bmatrix}
        \frac 12(\bar{P} - \partial_{\bar{z}}H) & 0 \\ 0 & \frac 12(\partial_{\bar{z}}H - \bar{P})
    \end{bmatrix} + \xi \begin{bmatrix}
        0 & 0 \\ P & 0
    \end{bmatrix}, \label{Lflatcon} \\[5pt]
    M &= \begin{bmatrix}
        0 & \bar{P} \\ 0 & 0
    \end{bmatrix} + \xi \partial_z + \xi \begin{bmatrix}
        \frac 12(\partial_z H - P) & 0 \\ 0 & \frac 12(P - \partial_z H) \nonumber 
    \end{bmatrix}.
\end{align}
Moreover in this case the QEE arise as a symmetry reduction of anti--self--dual Yang--Mills (ASDYM) equations
on $\R^4$ with gauge group $SU(2)$ by two translations.
\end{prop}
\begin{proof}
We have already shown the existence of the Lax pair. To establish the rest of the proposition
we follow \cite{MW, D24}.  Consider $\C^4$ equipped with the (complex) metric
\begin{equation*}
    \eta = 2(\text{d}z\text{d}\tilde{z} - \text{d}w\text{d}\tilde{w}).
\end{equation*}
Here $(z,\tilde{z},w,\tilde{w})$ are positively oriented complex coordinates on $\C^4$. Given a connection $\mathcal{A}$ on a vector bundle over $\C^4$, the ASDYM equations are
\begin{equation*}
    \star_{\eta} \mathcal{F} = -\mathcal{F}, \hspace{.4cm} \text{where} \hspace{.4cm} \mathcal{F} = \text{d}\mathcal{A} + \mathcal{A} \wedge \mathcal{A}.
\end{equation*}
These equations admit a Lax pair
\begin{equation} \label{laxasdym}
    \hat{L} = \mathcal{D}_{\tilde{z}} - \xi \mathcal{D}_w, \hspace{.8cm} \hat{M} = -\mathcal{D}_{\tilde{w}} + \xi \mathcal{D}_z,
\end{equation}
where $\mathcal{D} = \text{d} + \mathcal{A}$. To identify (\ref{laxasdym}) with (\ref{Lflatcon}) we choose Euclidean reality conditions $\tilde{z} = \bar{z}, \tilde{w} = -\bar{w}$ and set
\begin{equation} \label{asdymcon}
    \mathcal{A}_z = \begin{bmatrix}
        \frac 12(\partial_z H - P) & 0 \\ 0 & \frac 12(P - \partial_z H) 
    \end{bmatrix}, \hspace{.3cm} \mathcal{A}_{\bar{z}} = - (\mathcal{A}_z)^\dagger, \hspace{.3cm} \mathcal{A}_w = \begin{bmatrix}
        0 & 0 \\ -P & 0
    \end{bmatrix}, \hspace{.3cm} \mathcal{A}_{\bar{w}} = -(\mathcal{A}_w)^\dagger.
\end{equation}
Identifying $z$ with a complex coordinate on $U \subseteq M$, we can view $\mathcal{A}$ as a connection on $\pi^* (TU) \otimes \C$, where $\pi: U \times \R^2 \to U$ is the projection $\pi(z,\bar{z},w,\bar{w}) = (z,\bar{z})$. We now recognise the QEE as a symmetry reduction of the Euclidean ASDYM equations reduced by two translations $\partial_w, \partial_{\bar{w}}$ with gauge group $SU(2)$.
\end{proof}
\subsection{Hitchin's equations} 
\label{sub_hitch}
A symmetry reduction of the Euclidean ASDYM equations reduced by two translations $\partial_w, \partial_{\bar{w}}$  as in Proposition \ref{propasd} leads to Hitchin's (anti-)self-duality equations \cite{H87}. We shall look at this next,
and establish Theorem \ref{theoremHitchin} from the Introduction.
\begin{proof}[Proof of Theorem \ref{theoremHitchin}]
Given a solution to ASDYM, define
\begin{subequations}
\begin{align}
        A &= \mathcal{A}_z\text{d}z + \mathcal{A}_{\bar{z}}\text{d}\bar{z} = \begin{bmatrix}
        \frac 12(\partial_z H - P) & 0 \\ 0 & \frac 12(P - \partial_z H) 
    \end{bmatrix}\text{d}z + \begin{bmatrix}
        \frac 12(\bar{P}-\partial_{\bar{z}}H) & 0 \\ 0 & \frac 12(\partial_{\bar{z}} H - \bar{P}) 
    \end{bmatrix}\text{d}\bar{z}, \label{ac1}\\
    \Phi &= \mathcal{A}_{\bar{w}}\text{d}z = \begin{bmatrix} 0  & \bar{P} \\ 0 & 0 \end{bmatrix}\text{d}z.\label{phic1}
\end{align}
\end{subequations}
If we require our gauge transformations $\gamma$ to only depend on $z,\bar{z}$, the objects $A$ and $\Phi$ transform as 
\begin{equation}
\label{Hitchin_gauge}
    A \mapsto \gamma A \gamma^{-1} - \text{d}\gamma \cdot \gamma^{-1}, \hspace{.8cm} \Phi \mapsto \gamma \Phi \gamma^{-1}.
\end{equation}
Hence we may view $A$ as a $SU(2)$ connection (with respect to the $SU(2)$ structure defined by the hermitian metric $g_\C$ induced by $g$) on the complexified tangent bundle $E = TM \otimes \C$, and $\Phi$ as a (traceless) section of End$(E) \otimes \kappa$. Here $\kappa$ denotes the canonical bundle $\Omega^{1,0}M$. Coordinate-free expressions for the covariant derivative $\mathcal{D}_A$ and $\Phi$ are given by (\ref{t1}) and (\ref{t2}) 
respectively\footnote{A part of $\mathcal{D}_A$ is given by the Levi--Civita connection $\nabla$ of the metric $g$ in the quasi--Einstein structure. This is based on the identification of the holonomy group
$SO(2)$ of $(M, g)$ with $U(1)$, and embedding $U(1)$ in $SU(2)$:
\[
\begin{bmatrix} \cos{\theta}  & \sin{\theta} \\ -\sin{\theta} & \cos{\theta} \end{bmatrix}
\rightarrow e^{i\theta}\rightarrow 
\begin{bmatrix} e^{i\theta}  & 0 \\ 0& e^{-i\theta} \end{bmatrix}
\]
which yields the Lie--algebra map
\[
\begin{bmatrix} 0  & 1 \\ -1 & 0 \end{bmatrix}\rightarrow
\begin{bmatrix} i  & 0 \\ 0 & -i \end{bmatrix}.
\]
This explains the $H$--dependent part of the formula (\ref{ac1}).}. 

In particular, they show $\mathcal{D}_A$ and $\Phi$ are defined globally on $M$. Writing $F ~\in~\Omega^{1,1}(M ; \text{End }E)$ for the curvature of $A$ and $\Phi^*$ for the $g_\C$-adjoint of $\Phi$
\begin{equation*}
    \Phi^* = -\mathcal{A}_w\text{d}\bar{z} = \mathcal{P}^{0,1}\otimes (X^\flat)^{(1,0)} \in \Omega^{0,1}(M;\text{End } E),
\end{equation*}  
the ASDYM equations for $\mathcal{A}$ (and hence the QEE) are equivalent to the Hitchin equations (\ref{h11n}). Explicitly,
\begin{align*}
    F &= \text{d}A = \begin{bmatrix} \frac 12(\partial_z \bar{P} + \partial_{\bar{z}}P) - \partial_z\partial_{\bar{z}}H & 0 \\ 0 & -\frac 12(\partial_z \bar{P} + \partial_{\bar{z}}P) + \partial_z\partial_{\bar{z}}H \end{bmatrix}\text{d}z \wedge \text{d}\bar{z},\\
    [\Phi,\Phi^*] &= \begin{bmatrix} \vert P \vert^2 & 0 \\ 0 & -\vert P \vert^2 \end{bmatrix}\text{d}z\wedge \text{d}\bar{z}, \\
    \overline{\mathcal{D}}_A\Phi &= \overline{\partial}\Phi + [A^{(0,1)}, \Phi] = \begin{bmatrix} 0 & -\partial_{\bar{z}}\bar{P} + \bar{P}\partial_{\bar{z}}H - \bar{P}^2 \\ 0 & 0 \end{bmatrix}\text{d}z\wedge \text{d}\bar{z},
\end{align*}
from which we see that the equations (\ref{h11n}) (and their complex conjugate) are precisely (\ref{qeez}).
\end{proof} 
The second equation in (\ref{h11n}) is the statement that $\Phi$ is holomorphic with respect to the holomorphic structure on End$(E) \otimes \kappa$ induced by $A$ and the complex structure on $M$. Note that $A + \Phi + \Phi^*$ equals (the complex-linear extension of) the flat connection (\ref{D}). Indeed, flatness of this connection is a consequence of (\ref{h11n}).

To characterise the solution to the Hitchin equations corresponding to the QEE, recall that $A$ induces a holomorphic structure on $E$ by declaring trivialisations in which $A^{(0,1)}$ vanishes to be holomorphic. With respect to this structure $E$ splits holomorphically as $E = L \oplus L^{-1}$, where $L$ equals $T^{1,0}M$ as a smooth complex line bundle (but not necessarily as a holomorphic line bundle). This follows from the fact that, under the (smooth) isomorphism $T^{0,1}M \cong (T^*M)^{1,0}$ induced by $g$, $A$ is the direct sum of a $U(1)$ connection on $T^{1,0} M$ and its dual connection on $T^{0,1}M$. Note that the restriction 
\begin{equation*}
    \Phi\vert_{L^{-1}}: L^{-1}\otimes \kappa \to L \otimes \kappa 
\end{equation*}
is well-defined and hence holomorphic as a section of Hom$(L^{-1},L) \otimes \kappa$, which is canonically isomorphic to $L^2 \otimes \kappa$. This observation, together with Hitchin's vanishing theorem \cite[Theorem 2.1]{H87}, gives an alternative proof of the following known result (see \cite[Theorem 1.2]{CDKL24} or Theorem \ref{negmthm}).
\begin{prop} \label{l1}
    Let $(g,X)$ solve the QEE with $m = -1, \lambda = 0$ on an oriented compact 2-manifold $M$. Then $X \equiv 0$ and $(M,g)$ is the flat torus.
\end{prop}
\begin{proof}
    Note that\footnote{Here deg refers to the first Chern number of a smooth complex line bundle. For a compact Riemann surface of genus $\mathfrak{g}$ we have deg$(L) = -\text{deg}(\kappa) = 2-2\mathfrak{g}$.} deg$(L) = - \text{deg}(\kappa)$, so deg$(L^2 \otimes \kappa) = \text{deg}(L)$. Since the bundle $L^2 \otimes \kappa$ admits a holomorphic section $\Phi\vert_{L^{-1}}$, we must have deg$(L) \geq 0$.  On the other hand, the subbundle $L \subseteq E$ is $\Phi$-invariant, so Hitchin's vanishing theorem \cite{H87} implies deg$(L) \leq \frac 12\text{deg}(\wedge^2 E) = 0$. We find deg$(L) = 0$, i.e. $M$ is a torus. Moreover, the theorem in \cite{H87} tells us that $[\Phi,\Phi^*] = F = 0$, which translates into $X \equiv 0$ and $g$ being flat.
\end{proof}
As the equations are of interest only locally, from now on we restrict to an open subset of $M$ on which $X$ does not vanish. In this case $\Phi\vert_{L^{-1}}$ is nowhere zero and trivialises $L^2 \otimes \kappa$, so that we may view $L^{-1}$ as a square root of $\kappa$. Under this identification, we have
\begin{equation} \label{phiabstract}
    E = \kappa^{-\frac 12} \oplus \kappa^{\frac 12}, \hspace{.8cm} \Phi = \begin{bmatrix} 0 & \bf{id}\\ 0 & 0 \end{bmatrix}.
\end{equation}
Here $\bf{id}$ denotes the canonical section of Hom$(\kappa^{\frac 12}, \kappa^{-\frac 12}) \otimes \kappa \cong \kappa^{-1} \otimes \kappa$. We next show that $A$ corresponds to the connection on $\kappa^{-\frac 12} \oplus \kappa^{\frac 12}$ induced by the Chern connection $\nabla_h$ of the conformally rescaled metric~$h = \vert X \vert^2_g\:g$. We then recognise (\ref{phiabstract}) as the canonical solution to the $SU(2)$ Hitchin equations discussed in \cite[Example 1.5]{H87}.

\begin{prop} \label{hcc}
    Under the isomorphism provided by $\Phi\vert_{L^{-1}}$, the solution to the Hitchin equations in Theorem \ref{theoremHitchin} reduces to the canonical solution on $\kappa^{-\frac 12} \oplus \kappa^{\frac 12}$ defined (up to a homothety) by the metric $h = \vert X \vert^2\:g$ as in \textup{\cite{H87}}. In particular, $h$ has constant scalar curvature $-2$ whenever the QEE are satisfied.
\end{prop}
\begin{proof} It is immediate that $\Phi$ is given by (\ref{phiabstract}) as we are using $\Phi$ to identify $L^{-1}$ with $\kappa^{\frac 12}$. It remains to prove $A$ corresponds to the connection induced by $\nabla_h$. To see this, we work in the coordinates (\ref{coords}) in which $h = 4\vert P \vert^2\text{d}z\text{d}\bar{z}$. In the trivialisation of $\kappa$ defined by d$z$, the connection 1-form induced by $\nabla_h$ is $-\partial \log \vert P \vert^2$. We perform a gauge transformation $\gamma = \bar{P}^{-1}$ from this holomorphic trivialisation to the unitary trivialisation given by $\bar{P}\text{d}z$. The connection 1-form becomes 
\begin{equation*}
    \gamma (-\partial \log \vert P \vert^2)\gamma^{-1} - \text{d}\gamma \cdot \gamma^{-1} = \bar{\partial} \log \bar{P} - \partial \log P 
\end{equation*}
and so the corresponding connection 1-form on $\kappa^{-\frac 12} \oplus \kappa^{\frac 12}$ is
\begin{equation} \label{ah}
    A_h = \begin{bmatrix} \frac 12 P^{-1}\partial_z P & 0 \\ 0 & -\frac 12P^{-1}\partial_z P\end{bmatrix}\text{d}z + \begin{bmatrix} -\frac 12\bar{P}^{-1}\partial_{\bar{z}} \bar{P} & 0 \\ 0 & \frac 12\bar{P}^{-1}\partial_{\bar{z}} \bar{P}\end{bmatrix}\text{d}\bar{z}.
\end{equation}
Note that the section $\Phi\vert_{L^{-1}} \in \Gamma(\text{Hom}(L^{-1}, L)\otimes \kappa)$ equals $(e^{-\frac H2}\partial_z \otimes e^{\frac H2}\text{d}\bar{z}) \otimes \bar{P}\text{d}z$, and hence the isomorphism $L \cong \kappa^{-\frac 12}$ corresponds to the identity map in the trivialisations (\ref{ac1}, \ref{ah}). After using (\ref{zz}) and (\ref{zbzb}) we see that $A_h$ is equal to (\ref{ac1}), which proves our claim. The Higgs field $\Phi$ in (\ref{phiabstract}) is automatically holomorphic, and the remaining Hitchin equation is 
\begin{equation*}
\partial_z\partial_{\bar{z}}\:(\log \vert P \vert^2) = 2\vert P \vert^2,
\end{equation*}
which is the requirement that $h$ has constant negative curvature $R_h = -2$. \end{proof}

\subsection{Recovering the quasi-Einstein structure}
Let us locally introduce a complex function $Q$ such that $P = e^Q$. The fact that $h$ has constant curvature determines the real part of $Q$ as a solution to the Liouville equation
\begin{equation} 
    \partial_z\partial_{\bar{z}}(Q + \bar{Q}) = 2e^{Q + \bar{Q}}.
\end{equation}
To recover the quasi-Einstein structure, we still need to solve for the imaginary part of $Q$. This function satisfies an additional equation coming from the requirement that $\Phi$ is holomorphic, which is equivalent to (\ref{zz}) and (\ref{zbzb}) and has disappeared under the identification leading to (\ref{phiabstract}). Cross-differentiating (\ref{zz}) and (\ref{zbzb}) to obtain expressions for $\partial_z\partial_{\bar{z}} H$, we find the integrability condition
\begin{equation} 
\label{imQ}
    \partial_z\partial_{\bar{z}}(Q - \bar{Q}) = e^{\bar{Q}}\partial_z\bar{Q}-e^Q\partial_{\bar{z}}Q.
\end{equation}
Let us write $Q = \beta - i\Theta$ for real functions $\beta, \Theta$ and choose the coordinate $z = x + iy$ such that $\beta = -\log(2y)$ (i.e. the metric $h$ is in the upper half-plane model). Then (\ref{imQ}) is equivalent to
\[
    y^2(\partial_x^2\Theta + \partial_y^2\Theta) + y(\sin \Theta\: \partial_y \Theta + \cos \Theta\: \partial_x \Theta) + \cos \Theta = 0,
\]
which is the PDE (\ref{alphaeq}) from the Introduction.

If $\Theta$ satisfies (\ref{alphaeq}), we may recover $H$ and hence the full quasi-Einstein structure by integrating (\ref{zz}) or (\ref{zbzb}), which become
\be
\label{h1}
    \partial_x H = -\partial_y\Theta + \frac{\cos \Theta}{y}, \quad
    \partial_y H = \partial_x \Theta + \frac{\sin \Theta}{y} - \frac 1y. 
\ee
Note that for the projective structure it suffices to know d$H$ and no integration is needed. The integrability of equation (\ref{alphaeq}) is an interesting open problem. We have been unable to prove its equivalence with any of the known two-dimensional integrable systems. In particular, it does not appear to be of Liouville type in the sense of  Zhiber--Sokolov \cite{Sokolov}.

The solution (\ref{phiabstract}, \ref{ah}) is explicitly given on the upper half--plane $\{y > 0\}$  by
\be 
\label{hitchinexample}
A_0=\begin{bmatrix}i &0\\ 0 &-i\end{bmatrix}\frac{dx}{2y},\quad 
\Phi_0=\begin{bmatrix}0 &1\\ 0 &0\end{bmatrix}
\frac{dx+idy}{2y}, \quad h=\frac{dx^2+dy^2}{y^2}.
\ee
It  is gauge-equivalent to the solution (\ref{ac1}, \ref{phic1}): we find that $(A_0, \Phi_0)\rightarrow (A, \Phi)$
if we apply  the $SU(2)$ gauge transformation
(\ref{Hitchin_gauge}) with
\begin{equation*}
    \gamma = \begin{bmatrix} e^{\frac 12i\Theta} & 0 \\ 0 & e^{-\frac 12i\Theta} \end{bmatrix}
\end{equation*}
provided we set $P = \frac{1}{2y}e^{-i\Theta}$ and $\Theta$ satisfies (\ref{alphaeq}), so that we can define $H$ by (\ref{h1}). In particular, the solution to (\ref{alphaeq}) is needed to transform Hitchin's canonical example (Example 1.5 in \cite{H87})
into a gauge where we can read off a solution to the QEE.

\begin{example}
 Equation (\ref{alphaeq}) admits constant solutions $\Theta = \pm\tfrac 12\pi$. The corresponding quasi-Einstein structures have constant curvature and are given by
    \begin{alignat*}{2}
        g &= \text{d}x^2 + \text{d}y^2, \hspace{1cm} X^\flat = \frac 1y\text{d}y \hspace{.8cm} &&\text{ for $\Theta = \tfrac 12\pi$}, \\
        g &= \frac{1}{y^2}(\text{d}x^2 + \text{d}y^2), \hspace{.6cm} X^\flat = -\frac 1y\text{d}y \hspace{.6cm} &&\text{ for $\Theta = -\tfrac 12\pi$}.
    \end{alignat*}
More generally, solutions corresponding to symmetries of (\ref{alphaeq}) are presented in \cite{CDKL24} and Section \ref{solsm-1}.
\end{example}
\subsection{Hyper--K{\"a}hler metric}
There is a procedure \cite{DMW}  (see also \cite{DSG}, \cite{Hugo})  which associates a hyper--K\"ahler metric on an open set ${\mathcal{M}}$ of $\R^2\times S^2$ with any solution of ASDYM on $\R^4$ invariant under two translational symmetries. The idea is to 
identify generators of $\mathfrak{su}{(2)}$ with Hamiltonian vector--fields on $S^2$. This is done by considering $SU(2)$ as a subgroup of $SU(\infty)\cong \mbox{SDiff}(S^2)$, where the Hamiltonian action
is given by M\"obius transformations.
Performing this at the level of the Lax pair (\ref{laxasdym}, \ref{asdymcon}) yields a Lax pair consisting of volume--preserving vector
fields on $\R^2\times S^2$ and this, in turn, gives a null tetrad for a hyper--K\"ahler metric. 
The conformal class of this metric is defined by demanding that $L$ and $M$ in (\ref{Lflatcon}) span a $\CP^1$ worth
of $\alpha$--surfaces through each point of $\mathcal{M}$.
Implementing this procedure for (\ref{Lflatcon}) allows to absorb the unknown function $\Theta$
satisfying (\ref{alphaeq}) into coordinates,
so that the resulting metric only depends on Hitchin's canonical example (\ref{hitchinexample}). This metric
admits $\R\otimes\mathfrak{sl}{(2)}$ as its isometry algebra, with 
elements of $\mathfrak{sl}{(2)}$ generating tri--holomorphic isometries. The metric must therefore belong to the Gibbons--Hawking class \cite{GH}, and a further coordinate transformation identifies it as a limiting case of the double Eguchi--Hanson solution with opposite masses when the mass centres coincide. This is an example of a {\em folded} metric, and is in agreement
with the results of Hitchin \cite{Hitchin2} who  constructed the same metric starting from $SU(\infty)$ Higgs bundles.

\section{Solutions without Killing vectors} \label{swni}
Every known solution to the quasi-Einstein equations on a closed 2-manifold admits a Killing vector field. For $m \notin (0,2)$ a Killing vector necessarily exists in the compact case by Theorem \ref{RQEi}. In the case where $X$ is a gradient, the QEE can be solved even locally, and solutions again always admit a Killing vector.  In this section we show explicitly that there exist non-compact and non-gradient quasi-Einstein structures on a 2-manifold with only a discrete isometry group.
\subsection{General solution from a third order PDE}
We begin by introducing coordinates $(x,y)$ such that 
\begin{equation*}
    g = A(x,y)\text{d}x^2 + B(x,y)\text{d}y^2, \hspace{.8cm} X^\flat = C(x,y)\text{d}y
\end{equation*}
for some functions $A,B,C$. Note that this is always possible: up to scale there is a unique vector $Y$ such that $g(X,Y) = 0$, and locally we can always find non-vanishing functions $\nu_1,\nu_2$ such that $\nu_1X$ and $\nu_2Y$ commute by solving
\begin{equation*}
    Y(\nu_1) = \frac{g(X,[X,Y])}{\vert X \vert^2}\nu_1, \hspace{.8cm} X(\nu_2) = -\frac{g(Y,[X,Y])}{\vert Y \vert^2}\nu_2.
\end{equation*}
We then take $\nu_1X$ and $\nu_2Y$ as our coordinate vector fields $\partial_x$ and $\partial_y$. In these coordinates the $xy$-component of the QEE gives 
\begin{equation} \label{dphi}
   \partial_x\left(\frac{C}{B}\right) = 0.
\end{equation} Hence (after redefining $y$ if necessary) we may assume that $B = C$. Next we combine the $xx$ and $yy$ equations to eliminate the second derivative terms. We obtain
\begin{equation} \label{divphi}
    2B^2A+mB\partial_y A - mA\partial_y B = 0.
\end{equation}
Solving this equation for $A$ yields (after redefining $x$ if necessary) $A  = B\text{exp}(-\frac 2m U)$, where $U$ is a function such that $\partial_y U = B$. We find 
\begin{equation} \label{Uans}
    g = \partial_y U(x,y) \text{ exp}(-\tfrac 2m U(x,y)) \text{d}x^2 + \partial_y U(x,y)\text{d}y^2, \hspace{.8cm} X^\flat = \partial_y U(x,y) \text{d}y,
\end{equation}
Note that (\ref{dphi}) and (\ref{divphi}) are equivalent to Jezierski's vector field $\Phi = X/\vert X \vert^2$ satisfying d$\Phi^\flat = 0$ and div~$\Phi = -\frac 1m$ \cite{J09}. The final unsolved component of the QEE gives a PDE for $U$
\begin{align} \label{pdeU}
     me^{\tfrac 2mU}\left(mU_yU_{xxy}- mU_{xy}^2 + U_xU_yU_{xy}\right) + m^2U_yU_{yyy}-m^2U_{yy}^2&\nonumber \\-m(m+3)U_y^2U_{yy} + 2(m+1)U_y^4 + 2m^2\lambda U_y^3 = 0&.
\end{align}
Here $U_x = \partial_x U$ etc. We have thus reduced the QEE on a 2-manifold to a scalar third order PDE. It can be written invariantly as
\begin{equation*}
\Delta \log \vert X \vert^2 -\left(1 +\frac 2m\right)\text{div } X +\frac 1m \vert X \vert^2 + 2\lambda = 0,
\end{equation*}
which can be directly deduced from the QEE. For $m = -1$ and $\lambda = 0$ this equation together with the trace of the QEE implies
\begin{equation*}
    \Delta \log \vert X \vert^2 = R + 2\vert X \vert^2,
\end{equation*}
which is equivalent to the requirement that the metric $h = \vert X \vert^2 g = (U_y)g$ has constant scalar curvature equal to $-2$ (see Proposition \ref{hcc}).

\subsection{Quasi-Einstein 2-manifolds with a homothety} \label{homsec}
Let us consider the group invariant solutions corresponding to symmetries of (\ref{pdeU}). For general $\lambda$ these include coordinate transformations in $x$ and translations in $y$. However, when $\lambda = 0$ there is an additional scaling symmetry generated by 
\begin{equation} \label{homothety}
    K = x\partial_x + y\partial_y.
\end{equation}
The group invariant solutions with $U = U(y/x)$ satisfy $\mathcal{L}_K X^\flat = 0$ and $\mathcal{L}_K g = g$, i.e. $K$ is a homothety. In fact, it follows from the QEE that the homothety must preserve $X^\flat$.
\begin{lemma} \label{KX}
    Let $(g,X)$ solve the QEE with $\lambda = 0$ on a connected 2-manifold admitting a (non-Killing) homothety~$K$. Then either $\mathcal{L}_K X^\flat = 0$ or $g$ is flat.
\end{lemma}
In the flat case one can always find a (possibly different) homothety preserving $X^\flat$. The proof of Lemma \ref{KX} uses a prolongation argument following \cite[Proposition 3.1]{CDKL24}. For the details of this argument we refer the reader to Appendix \ref{Ais}.

We are now in a position to find the general quasi-Einstein structure with $\lambda = 0$ admitting a homothety. It is a 2-parameter family of solutions for any $m$ that do not admit Killing vectors unless $X^\flat$ is closed.

\begin{prop} \label{hvf}
    Let $(g,X)$ solve the QEE with $\lambda = 0$ on a 2-manifold admitting a (non-Killing) homothety $K$. Then there exist local coordinates $(r,s)$ such that 
    \begin{subequations} \label{gxrs}
    \begin{align}
        g &= e^r Z(s)\left((1+s^2)\textup{d}r^2 + 2(s+V(s))\textup{d}r\textup{d}s + (V(s)^2 + 1)\textup{d}s^2\right), \\ X^\flat &= Z(s)(s\textup{d}r + \textup{d}s).
    \end{align}
    \end{subequations}
    Here $V(s) = Z(s)(m+sZ(s))^{-1}$ and $Z(s) = e^{r}\vert X \vert^2$ satisfies the second order ODE
    \begin{align} \label{ode2}
        Z(s^2+1)(m+sZ)^2Z''&-m(s^2+1)(m+Zs)(Z')^2 +ZZ'(m+Zs)(2ms-mZ+4s^2Z) \nonumber \\
        &+Z^2(m^2+3msZ+Z^2(2s^2+m)) = 0.
    \end{align}
\end{prop}
\begin{proof}
By Lemma \ref{KX} we may assume that $\mathcal{L}_K g = g$ and $\mathcal{L}_K X^\flat = 0$. We claim that there exist coordinates $(x,y)$ as in (\ref{Uans}) such that $K$ takes the form (\ref{homothety}). Indeed, let us write 
\begin{equation*}
    K = \alpha(x,y)\partial_x + \beta(x,y)\partial_y
\end{equation*} 
and $\Phi = X/\vert X \vert^2$. From $\mathcal{L}_K \Phi^\flat = \Phi^\flat$ and the fact that $\Phi^\flat = \text{d}y$ in the coordinates (\ref{Uans}) we find (after translating $y$) that $\beta = y$. In order for $\mathcal{L}_K g$ to be diagonal we require $\partial_y \alpha = 0$. If $\alpha \equiv 0$ then $U = -m\log y + a(x)$ for some function $a$, which is incompatible with (\ref{pdeU}). Hence $\alpha$ must be non-zero, so that we can redefine $x$ to obtain (\ref{homothety}).

To find the group invariant solutions, we introduce coordinates $(v,w)$ by $v = \frac yx$ and $w = \log x$. Then $U = U(v)$ depends only on $v$, and we have
\begin{equation*} 
    g = e^wU'(v)\left((v^2 +e^{-\frac 2mU(v)})\text{d}w^2 + 2v\text{d}v\text{d}w + \text{d}v^2\right), \hspace{.8cm} X^\flat = U'(v)(\text{d}v + v\text{d}w).
\end{equation*}
The function $U$ satisfies a third order ODE coming from (\ref{pdeU})
\begin{align*}
    m e^{\frac 2mU}\left(mv^2 U'''U'\right.&\left.-mv^2(U'')^2 + v^2U''(U')^2+2mvU'U''+ m(U')^2+v(U')^3\right) \\
    &+ m^2 U'''U' - m^2(U'')^2-m(m+3)(U')^2U''+2(m+1)(U')^4 = 0.
\end{align*}
This reduces to the second order equation (\ref{ode2}) for $Z = e^{-\frac{U(v)}{m}}U'(v)$ in the coordinates $(r,s)$ defined by 
\begin{equation*}
r = w-\tfrac 1m U(v), \hspace{.8cm} s = e^{\frac 1mU(v)}v.
\end{equation*}
Note that the coordinate transformation is well-defined since there are no solutions to the ODE for $U$ with $U(v) = -m\log v \:+\: $const (for which $s$ would be constant). The data $(g,X)$ in the new coordinates becomes (\ref{gxrs}).
\end{proof}

\subsection{Solutions with $m = -1$} \label{solsm-1} For the case $(m = -1,\lambda = 0)$ the solutions in Proposition \ref{hvf} can be written down explicitly using the observation that the metric $\vert X \vert^2 g$ has constant curvature.

\begin{prop} \label{hvfexp}
    Let $(g,X)$ solve the QEE with $m = -1$ and $\lambda = 0$ on a 2-manifold admiting a (non-Killing) homothety $K$. Then there exist local coordinates $(x,y)$ such that 
\begin{subequations}
\begin{align}
    g &= e^y C'(x)\left(\frac{1}{x-\alpha}\textup{d}x^2 + x^2\textup{d}y^2\right), \label{gexp}\\
    X^\flat &= \frac{1}{2xC'(x)}\left((3C'(x)+2xC''(x))\textup{d}x+ (2xC'(x)+C(x))\textup{d}y\right), \label{Xexp}
\end{align}
\end{subequations}
for some $\alpha \in \R$. The function $C(x)$ satisfies the ODE $L_\alpha C = 0$, with
\begin{equation} \label{ode2C}
L_\alpha C = x^4(x-\alpha)(C'')^2+3x^3(x-\alpha)C'C''+x^2(2x+1-\tfrac 94\alpha)(C')^2+xCC'+\tfrac{1}{4}C^2.
\end{equation}
\end{prop}
We require $x > \text{min}(0,\alpha)$ to ensure that (\ref{gexp}) is a well-defined Riemannian metric. The equation $L_\alpha C = 0$ can be solved using the observation 
\begin{equation*}
    \frac{\text{d}}{\text{d}x}\left[L_\alpha C(x)\right] = (3C'+ 2xC'')D_\alpha C(x).
\end{equation*}
Here $D_\alpha$ is the third order linear operator
\begin{equation} \label{Dalpha}
    D_\alpha C(x) = x^3(x-\alpha)C'''+(4x^3-\tfrac 72\alpha x^2)C''+x(2x+1-\tfrac 32\alpha)C'+\tfrac 12C.
\end{equation}
The solution to (\ref{Dalpha}) in terms of generalised hypergeometric functions is described in Appendix \ref{Ab}.

\begin{proof}[Proof of Proposition \ref{hvfexp}]
By Lemma \ref{KX} we have $\mathcal{L}_K X^\flat = 0$. Hence $K$ is a Killing vector of the metric $h = \vert X \vert^2 g$. Moreover, $h$ has constant curvature $-2$ by Proposition \ref{hcc}. Set $x =4h(K,K)$ (note that $x$ is not constant) and introduce a second coordinate $y$ in such a way that $K = \partial_y$ and $h$ is diagonal in the coordinates $(x,y)$. We then have 
\begin{equation*}
    h = E(x)\text{d}x^2 + \tfrac 14 x\text{d}y^2.
\end{equation*}
The equation $R_h = -2$ reads
\begin{equation*}
    \frac{E+xE'}{2E^2x^2} = -2,
\end{equation*}
with solution
\begin{equation*}
    E(x) = \frac{1}{4x(x-\alpha)},
\end{equation*}
where $\alpha \in \R$. Next we introduce functions $A,B$ by
\begin{equation*}
    A = e^{-y}\vert X \vert^{-2}, \hspace{1cm} B = e^{-y}f.
\end{equation*}
Here $f$ is a Kähler potential satisfying d$f = \Phi^\flat = X^\flat/\vert X \vert^2$ (recall that $\Phi^\flat$ is closed). Note that $A$ is independent of $y$ since $K$ preserves $X^\flat$, and $f$ must satisfy d$(\mathcal{L}_Kf - f) = 0$. Hence we may add a constant to $f$ so that $B = B(x)$ is also a function of $x$ only. In the coordinates $(x,y)$ the data $(g,X)$ now takes the form
\begin{equation} \label{ansab}
    g = \frac 14e^yA(x)\left(\frac{1}{x(x-\alpha)}\text{d}x^2 + x\text{d}y^2\right),\hspace{.8cm}
    X^\flat = \frac{1}{A(x)}(B'(x)\text{d}x + B(x)\text{d}y).
\end{equation}
It remains to solve for $A$ and $B$. The $xy$-component of the QEE gives
\begin{equation*}
    2xA'-2xB'+A= 0.
\end{equation*}
This can be solved by introducing a function $C$ such that $A = 2xC'$ and $B = C + 2xC'$. In terms of $C$ our Ansatz (\ref{ansab}) becomes (\ref{gexp}) and (\ref{Xexp}) after a shift in $y$. Next we combine the $xx$- and $yy$-components to eliminate third order derivatives of $C$. This results in the equation $L_\alpha C(x) = 0$ with $L_\alpha$ given by~(\ref{ode2C}). It is now straightforward to check that the QEE hold as a consequence of the vanishing of~(\ref{ode2C}) and (\ref{Dalpha}).
\end{proof}

\begin{example}
   As explained in Appendix \ref{Ab}, when $\alpha^{-1}$ is a square integer the quasi-Einstein structure can be written down explicitly without using hypergeometric functions. For example, for $\alpha = 1$ the solution to $D_\alpha C = 0$ in terms of arbitrary constants $a,b,c$ is
\begin{equation} \label{Csol}
     C(x) = \tfrac{a + b\left(\sqrt{x(x-1)}-\arcsinh(\sqrt{x-1})\right) + c\left(3x+\arcsinh(\sqrt{x-1})^2-2\sqrt{x(x-1)}\arcsinh(\sqrt{x-1})\right)}{x}.
\end{equation}
Inserting (\ref{Csol}) into (\ref{ode2C}) gives the constraint $b^2 = 4ac$. We assume $c < 0$ to ensure $g$ is positive definite and $X^\flat$ is not closed. By shifting $y$ and rescaling $C$ we can then set $c = -1$. To solve the constraint we take $b = 2\beta$ and $a = \beta^2$. We are left with a 1-parameter family of solutions, which can be written in terms of the coordinate $t = \sqrt{x - 1}$ as
\begin{subequations}
\begin{align}
    g &= e^y(f(t)^2 + t^2 + 1)\left(\frac{4}{(1+t^2)^2}\text{d}t^2 + \text{d}y^2\right), \\[8pt]
    X^\flat &=  \frac{- tf(t)^2+2f(t)\sqrt{t^2+1}+t(t^2+1)}{(t^2+1)(f(t)^2 + t^2 + 1)}\text{d}t+\frac{f(t)^2 + 2tf(t)\sqrt{t^2+1}- t^2-1}{2(f(t)^2 + t^2 + 1)}\text{d}y,
\end{align}
\end{subequations}
where $f(t) = \arcsinh(t) + \beta$.
\end{example}

\appendix

\section{Details of proofs in Section \ref{qeerig}.} \label{Ac}
In this Appendix we present the proof of the tensor identities in Lemma \ref{lemid} and Proposition~\ref{propid}.

\begin{proof}[Proof of Lemma \ref{lemid}] Note that the QEE \eqref{QEE} are equivalent to \eqref{QEEK} after using the definition \eqref{Kansatz} of $K$. By taking the trace of (\ref{QEEK}), we obtain the scalar curvature
\begin{equation} \label{trqee}
    R = \frac{m}{4\Gamma^2}\vert K \vert^2 - \frac{m}{2\Gamma}\nabla_c K^c + \frac{m}{2\Gamma}\Delta \Gamma - \frac{m}{4\Gamma^2}\vert \nabla \Gamma \vert^2 + n\lambda.
\end{equation}
From (\ref{QEEK}) and (\ref{trqee}) it follows that
\begin{align*}
\nabla^aR_{ab} &= \frac{m}{4\Gamma^2}(\nabla^a K_a)K_b + \frac{m}{4\Gamma^2}K^a\nabla_aK_b - \frac{m}{2\Gamma^3}(K^a\nabla_a \Gamma)K_b - \frac{m}{2\Gamma}\nabla^a\nabla_{(a}K_{b)} + \frac{m}{2\Gamma^2}(\nabla^a\Gamma)\nabla_{(a}K_{b)} \\[5pt]
&+ \frac{m}{2\Gamma}\nabla^a\nabla_a\nabla_b\Gamma - \frac{m}{2\Gamma^2}(\nabla^a\Gamma )\nabla_a\nabla_b \Gamma - \frac{m}{4\Gamma^2}(\Delta \Gamma)\nabla_b\Gamma - \frac{m}{4\Gamma^2}(\nabla^a \Gamma)\nabla_a\nabla_b\Gamma + \frac{m}{2\Gamma^3}\vert \nabla \Gamma \vert^2\nabla_b\Gamma , \\[5pt]
\frac 12\nabla_b R &= \frac m4\nabla_b\left(\frac{1}{2\Gamma^2}\vert K \vert^2 - \frac{1}{2\Gamma^2}\vert \nabla \Gamma \vert^2 - \frac{1}{\Gamma}\nabla_c K^c + \frac{1}{\Gamma}\Delta \Gamma\right).
\end{align*}
The difference of these expressions vanishes as a consequence of the contracted Bianchi identity $\nabla^a(R_{ab} - \frac 12Rg_{ab})$. After multiplying by $\Gamma$, we may isolate the terms involving $\nabla_{(a}K_{b)}$ in this identity. The result is independent of $m$ and given by (\ref{idstep1}).
\end{proof}
\noindent From the identity (\ref{idstep1}) we deduce Proposition \ref{propid} following the method outlined in Section \ref{tensoridsec}.
\begin{proof}[Proof of Proposition \ref{propid}]
    The Ricci identity and (\ref{QEEK}) imply
\begin{align*}
    \nabla^a\nabla_b\nabla_a\Gamma = \nabla_b \Delta \Gamma + R_{ab}\nabla^a\Gamma = &\nabla_b \Delta \Gamma + \frac{m}{4\Gamma^2}(K^a\nabla_a \Gamma)K_b -\frac{m}{2\Gamma}(\nabla^a\Gamma) \nabla_{(a}K_{b)}\\& + \frac{m}{2\Gamma}(\nabla^a\Gamma)\nabla_a\nabla_b\Gamma - \frac{m}{4\Gamma^2}\vert \nabla \Gamma \vert^2\nabla_b\Gamma + \lambda\nabla_b \Gamma.
\end{align*}
Hence (\ref{idstep1}) becomes 
\begin{align*}
    \nabla^a\nabla_{(a}K_{b)} + &(m-2)\frac{1}{2\Gamma}\nabla^a\Gamma\nabla_{(a}K_{b)} = \\
    &\frac{1}{2\Gamma}K_b\nabla_cK^c + \frac{1}{2\Gamma}K^a\nabla_aK_b + (m-4)\frac{1}{4\Gamma^2}(K^c\nabla_c\Gamma)K_b - \frac{1}{2\Gamma}(\Delta \Gamma) \nabla_b\Gamma \\[3pt]&+ (m-3)\frac{1}{2\Gamma}\nabla^a\Gamma\nabla_a\nabla_b\Gamma -(m-4)\frac{1}{4\Gamma^2}\vert \nabla \Gamma \vert^2\nabla_b\Gamma + \nabla_b\Delta \Gamma + \lambda \nabla_b\Gamma \\[3pt]
    &-\frac 12\Gamma \nabla_b\left(\frac{1}{2\Gamma^2}\vert K \vert^2 - \frac{1}{2\Gamma^2}\vert \nabla \Gamma \vert^2 - \frac{1}{\Gamma}\nabla_c K^c + \frac{1}{\Gamma}\Delta \Gamma\right) \\[3pt]
    &= \frac{1}{2\Gamma}K_b\nabla_c K^c + \frac{1}{2\Gamma^2}\vert K \vert^2\nabla_b \Gamma + \frac{1}{2\Gamma}K^a\nabla_aK_b - \frac{1}{4\Gamma}\nabla_b \left(\vert K \vert^2\right)
\\[3pt] &+ (m-4)\frac{1}{4\Gamma^2}(K^c\nabla_c\Gamma)K_b + \frac 12\nabla_b\Delta \Gamma + (m-2)\frac{1}{4\Gamma}\nabla_b \left(\vert \nabla \Gamma \vert^2\right)\\[3pt] &- (m-2)\frac{1}{4\Gamma^2}\vert \nabla \Gamma \vert^2\nabla_b \Gamma + \lambda \nabla_b \Gamma + \frac 12\Gamma\nabla_b\left(\frac{1}{\Gamma}\nabla_c K^c\right).
\end{align*}
In the final step we expanded the brackets on the fourth line and used
\begin{equation*}
    \nabla_b \left(\vert \nabla \Gamma \vert^2\right) = 2(\nabla^a\Gamma)\nabla_a\nabla_b \Gamma.
\end{equation*}
Next we contract with $K^b$ and multiply by $\Gamma^{\frac m2 -1}$. Using $\nabla_b (\vert K \vert^2) = 2K^a\nabla_bK_a$, we arrive at
\begin{align}
    \Gamma&^{\frac m2-1}K^b\nabla^a\nabla_{(a}K_{b)} + \frac 12(m-2)\Gamma^{\frac m2-2}\nabla^a\Gamma K^b\nabla_{(a}K_{b)} = \nonumber\\[3pt] &\frac 12\Gamma^{\frac m2-2}\vert K \vert^2\nabla_cK^c + \frac 14(m-2)\Gamma^{\frac m2-3}\vert K \vert^2K^a\nabla_a\Gamma
    + \frac 12\Gamma^{\frac m2-1}K^a\nabla_a \Delta \Gamma  + \lambda \Gamma^{\frac m2-1}K^a\nabla_a\Gamma \nonumber\\[3pt] &+ \frac 14(m-2)\Gamma^{\frac m2-2}K^a\nabla_a(\vert \nabla \Gamma \vert^2) - \frac 14(m-2)\Gamma^{\frac m2-3}\vert \nabla \Gamma \vert^2 K^a\nabla_a\Gamma + \frac 12\Gamma^{\frac m2}K^a\nabla_a(\Gamma^{-1}\nabla_cK^c). \label{idstep2}
\end{align}
Note that many cancellations occur precisely when $m = 2$. It remains to group terms into divergences and terms proportional to $\nabla_a(\Gamma^{\frac m2-1}K^a)$. The LHS can be expressed as
\begin{align} \label{lhsdiv}
    \Gamma^{\frac m2-1}K^b\nabla^a\nabla_{(a}K_{b)} + \frac 12(m-2)\Gamma^{\frac m2-2}&\nabla^a\Gamma K^b\nabla_{(a}K_{b)} = \nonumber\\ &\nabla^a\left(\Gamma^{\frac m2-1}K^b\nabla_{(a}K_{b)}\right) - \Gamma^{\frac m2-1}\nabla^{a} K^{b}\nabla_{(a}K_{b)}.
\end{align}
For the cubic terms in $K$ on the RHS of (\ref{idstep2}), we can write
\begin{equation} \label{a1}
    \frac 12\Gamma^{\frac m2-2}\vert K \vert^2\nabla_a K^a + \frac 14(m-2)\Gamma^{\frac m2-3}\vert K \vert^2K^a\nabla_a\Gamma = \frac{1}{2\Gamma}\vert K \vert^2\nabla_a\left(\Gamma^{\frac m2-1}K^a\right).
\end{equation}
For $m \neq 2$ the quadratic term in $K$ becomes
    \begin{align} \label{a2}
        \frac 12\Gamma^{\frac m2}K^a\nabla_a\left(\frac{1}{\Gamma}\nabla_cK^c\right) = &\nabla_a\left(\frac 12\Gamma^{\frac m2 - 1}(\nabla_cK^c)K^a\right)\nonumber \\
        &- \frac{m}{2(m-2)}(\nabla_c K^c)\nabla_a\left(\Gamma^{\frac m2 -1}K^a\right) + \frac{1}{m-2}\Gamma^{\frac m2 -1}(\nabla_cK^c)^2.
    \end{align}
Observe that this introduces an extra term proportional to $(\nabla_c K^c)^2$. For the term involving $\Delta \Gamma$, we have 
\begin{equation} \label{a3}
    \frac 12\Gamma^{\frac m2-1}K^a\nabla_a \Delta \Gamma = \nabla_a\left(\frac 12\Gamma^{\frac m2-1}(\Delta \Gamma)K^a\right) - \frac 12(\Delta \Gamma)\nabla_a\left(\Gamma^{\frac m2-1}K^a\right).
\end{equation}
Similarly, the term proportional to $\lambda$ is
\begin{equation} \label{a4}
    \lambda \Gamma^{\frac m2-1}K^a\nabla_a\Gamma = \nabla_a\left(\lambda\Gamma^{\frac m2}K^a\right) - \lambda \Gamma \nabla_a\left(\Gamma^{\frac m2-1}K^a\right).
\end{equation}
The remaining two terms involving $\vert\nabla \Gamma\vert^2$ can be manipulated using
\begin{equation} \label{a5}
    \Gamma^{\frac m2-2}K^a\nabla_a(\vert \nabla\Gamma\vert^2) - \Gamma^{\frac m2-3}\vert \nabla \Gamma \vert^2K^a\nabla_a\Gamma = \nabla_a\left(\Gamma^{\frac m2-2}\vert \nabla \Gamma \vert^2K^a\right) - \frac{\vert \nabla \Gamma \vert^2}{\Gamma}\nabla_a\left(\Gamma^{\frac m2-1}K^a\right).
\end{equation}
Combining (\ref{idstep2}) with \eqref{lhsdiv}--\eqref{a5} leads to the desired identity (\ref{identity}). 
\end{proof}

\section{Prolongation formulae} \label{Aa}

For a fixed metric $g$ and function $\Lambda$, the generalised extremal horizon equations (\ref{gehe}) form an overdetermined system for $X$. This system is not yet closed, as the equations only prescribe the symmetrized part of $\nabla X^\flat$. In the two-dimensional case, the system closes after introducing the function $\Omega = \star \text{d} X^\flat$. The closed system and first three algebraic constraints in $(X,\Omega)$ coming from the prolongation procedure were computed for the quasi-Einstein equations in \cite{CDKL24}. For the generalised extremal horizon equations, the first algebraic constraint is presented in Lemma \ref{prolonglem} and used in the proof of Theorem \ref{toprig}. Below we prove Lemma \ref{prolonglem} and recall the constraints computed in \cite{CDKL24}, which are required for the proof of Lemma \ref{KX}.

\begin{proof}[Proof of Lemma \ref{prolonglem}]
We will make use of abstract index notation for the proof. Equation (\ref{prolong}) can then be written as
\begin{equation} \label{pAIN}
    2\nabla_a X_b = -cX_aX_b - \Lambda g_{ab} + \Omega \epsilon_{ab}.
\end{equation}
We apply $\nabla_d$ to (\ref{pAIN}) and skew over $a,d$ using 
\begin{equation*}
    (\nabla_a\nabla_d - \nabla_d\nabla_a)X_b = \frac R2(g_{ab}X_d - g_{bd}X_a).
\end{equation*}
The result is
\begin{equation*}
    RX_{[a}g_{d]b} = -\frac 12c\Omega\epsilon_{da}X_b - cX_{[a}\nabla_{d]}X_b -(\nabla_{[d}\Lambda )g_{a]b} + (\nabla_{[d}\Omega)\epsilon_{a]b}.
\end{equation*}
The derivative of $X$ can be eliminated by applying (\ref{pAIN}) again,
\begin{equation*}
    RX_{[a}g_{d]b} = -\frac 12c\Omega\epsilon_{da}X_b +\frac 12c\Lambda X_{[a}g_{d]b} - \frac 12c\Omega X_{[a}\epsilon_{d]b}-(\nabla_{[d}\Lambda )g_{a]b} + (\nabla_{[d}\Omega)\epsilon_{a]b}.
\end{equation*}
Next we contract with $\epsilon^{da}$. Using $\epsilon_{ab}\epsilon^{bd} = -\delta_a^{\:d}$, we find
\begin{equation} \label{step1eps}
    R\epsilon^{\:d}_bX_d = -\frac 32c\Omega X_b + \frac 12c\Lambda\epsilon^{\:d}_{b}X_d  + \epsilon^{\:d}_{b}\nabla_d\Lambda - \nabla_b \Omega.
\end{equation}
Solving this expression for $\nabla_b\Omega$ yields (\ref{domega}) with $(\star X)_b = \epsilon_b^{\:d}X_d$. To go further, we differentiate (\ref{step1eps}) by applying $\epsilon^{ba}\nabla_a$. This leads to
\begin{equation*}
    (\tfrac 12c\Lambda-R)\nabla_aX^a + X^a\nabla_a(\tfrac 12c\Lambda-R) + \frac 32c\Omega^2 - \frac 32cX_b\epsilon^{ba}\nabla_a\Omega + \Delta \Lambda = 0.
\end{equation*}
The divergence of $X$ can be computed from the generalised extremal horizon equations and equals $-\frac c2\vert X \vert^2 - \Lambda$. Similarly, the derivative of $\Omega$ can be eliminated by applying (\ref{step1eps}) again. We obtain the following expression for the Laplacian of $\Lambda$
\begin{equation*}
    \Delta \Lambda = -\frac{3}{2}c\Omega^2 - X^a\nabla_a(2c\Lambda - R) + (\tfrac 12c\Lambda-R)(\Lambda -c\vert X \vert^2),
\end{equation*}
which is precisely (\ref{laplam})
\end{proof}
For the quasi-Einstein equations, the closed system becomes
\begin{subequations} \label{closedtot}
\begin{align}
\nabla X^\flat &= \frac 1m X^\flat \otimes X^\flat + (\lambda - \tfrac 12 R)g + \frac 12\Omega\mbox{vol}_g, \label{closed1} \\
\text{d}\Omega &= \frac 3m \Omega X^\flat + \star \text{d} R + \frac{1}{m}\left(2\lambda - \left(m + 1\right)R\right)\star\!X^\flat. \label{closed2}
\end{align}
\end{subequations}
The first three algebraic constraints derived from (\ref{closedtot}) in \cite{CDKL24} are
\begin{subequations}
\begin{align} 
0 &= -\Delta R + \left(1 + \frac 4m\right)\langle X,\text{d} R\rangle + \frac{3}{m}\Omega^2 + \frac{1}{m^2}(2\lambda - (m+1)R)(2\vert X \vert^2 -2\lambda m + mR) \; , \label{pqe1} \\
0 = &-\text{d}(\Delta R) + \frac 1m\left(1 + \frac 4m\right)\langle X, \text{d} R\rangle X^\flat + \left(1 + \frac 4m\right)\text{Hess}(R)(X) + \left(\frac 12 + \frac{8}{m}\right)\Omega\star\!\text{d} R \nonumber \\
&\left[- \frac 2m\left(1 + \frac 1m\right)\vert X \vert^2 - \left(\frac 52 + \frac 4m\right)R + \lambda\left(3 + \frac 8m\right)\right] \text{d} R + \frac{18}{m^2}\Omega^2X^\flat \nonumber \\
&+ \frac{4}{m^2} \left(2\lambda - (m+1)R\right)\left(\left(\lambda - \tfrac 12R\right)X^\flat + 2\Omega\star\! X^\flat + \frac 1m\vert X \vert^2X^\flat\right). \label{pqe2}
\end{align}
\end{subequations}

\section{Inheritance of symmetry} \label{Ais}

The constraints (\ref{pqe1}) and (\ref{pqe2}) were used in \cite[Proposition 3.1]{CDKL24} to show that if $(g,X)$ satisfies the QEE and $g$ admits a Killing vector $K$, then either $g$ has constant curvature or $K$ preserves $X$. Below we extend this result to the case where $K$ is a homothety satisfying $\mathcal{L}_Kg = g$, as considered in Section \ref{homsec}.

\begin{proof}[Proof of Lemma \ref{KX}] Let us introduce coordinates $(x,y)$ such that $K = \partial_x$ and 
\begin{equation*}
    g = e^x(\text{d}y^2 + k(y)^2\text{d}x^2), \hspace{.8cm} X^\flat = X_1(x,y)\text{d}x + X_2(x,y)\text{d}y.
\end{equation*}
Our aim is to show that $\partial_x X_1 = \partial_x X_2 = 0$ unless $g$ is flat. For this we will use the three algebraic constraints on $(X_1,X_2,\Omega)$ coming from (\ref{pqe1}) and (\ref{pqe2}). We set 
\begin{equation*}
    r = Re^x, \hspace{.6cm} p = (\Delta R)e^{2x}, \hspace{.6cm} \omega = \Omega e^{x},
\end{equation*}
and note that $p$ and $r$ are functions of $y$ only. The components of (\ref{pqe1}) and (\ref{pqe2}) (with $\lambda = 0$) read
\begin{subequations}
\begin{align}
    0 &= -p+\left(1 + \tfrac 4m\right)(X_2 r'-rk^{-2}X_1)+\tfrac 3m\omega^2-\tfrac{m+1}{m^2}r(2(k^{-2}X_1^2 + X_2^2) + mr), \label{w2}\\
    0 &= mkr'\left(m(m+4)X_1k'-\tfrac 12k\left((m+4)(3m-2X_1)X_2-m(m+16)k\omega\right)\right)+2m^3pk^2 \nonumber \\ &\:\:+ 18m\omega^2k^2+rX_1(\tfrac 32m^3+m^2(X_2+6)-2mX_1(2X_1+1)-4X_1^2)+ m^2(m+4)krX_2k' \nonumber \\
    &\:\:+mr^2k^2\left(\tfrac 52m^2+2m(X_1+2)+2X_1\right)+2(m+1)X_2^2(m-2X_1)rk^2-8m(m+1)rX_2k^3\omega, \label{xcomp}\\
    0 &= m^2(m+4)(k^3X_2r''+rX_1k')-mr'X_1k(\tfrac 32m^2+2m(X_1+3)+2X_1)-m(m-2)k^3r'X_2^2 \nonumber\\
    &\:\:-m^3k^3p' - m^2krr'(\tfrac 52m + 4)-4(m+1)X_2^3k^3r+18mX_2k^3\omega^2 -m^2krX_2(X_1+2)-\tfrac 12m^3krX_2\nonumber \\
    &\:\:+2m(m+1)X_2k^3r^2+mk^2\omega r(\tfrac 12m^2+8m(X_1+1)+8X_1)-4krX_2((X_1^2+X_1)m+X_1), \label{ycomp}
\end{align}
\end{subequations}
where the prime denotes a $y$-derivative. As in \cite[Proposition 3.1]{CDKL24} it suffices to prove that $\partial_x X_1 = 0$ and $ \partial_x X_2 = 0$ on any open set where $R$ is not constant (note that if $R$ is constant, $g$ must be flat). Let us first consider the generic case where $m \neq -1$ and $m \neq -4$. We solve (\ref{w2}) for $\omega^2$ and use this to eliminate all instances of $\omega^2$ in (\ref{xcomp}) and (\ref{ycomp}). We can then solve (\ref{xcomp}) for $\omega f$, where
\begin{equation*}
    f = m(m+16)r'-16(m+1)rX_2.
\end{equation*}
Using this to eliminate all remaining instances of $\omega$ from (\ref{w2}) times $f^2$ and (\ref{ycomp}) times $f$, we obtain two polynomial constraints $P_1(X_1,X_2)$ and $P_2(X_1,X_2)$ respectively. $P_1$ has degree 6 in $X_1$ and degree 4 in $X_2$, whereas $P_2$ has degree 4 in both $X_1$ and $X_2$. The coefficients of $P_1,P_2$ depend on $y$ only. Explicitly,

\begin{align*}
    P_1(X_1,X_2) &= X_1^6+k^4\left((X_1+\tfrac m4)^2 - \tfrac 23m(m+1)rk^2\right)X_2^4 + \dots, \\
    P_2(X_1,X_2) &= X_1^4+k^4X_2^4 + \dots,
\end{align*}
where the dots denote terms of lower order in either $X_1$ or $X_2$. To eliminate $X_1$ we compute the resultant res$(P_1,P_2)$, viewing the $P_i$ as polynomials in $X_1$. We obtain a polynomial constraint of degree 16 in $X_2$ only with coefficients depending on $y$. If $X_2$ depends on $x$, this polynomial must be identically zero. Setting the leading coefficient to zero gives
\begin{equation} \label{squares}
    \left(4r''rk^2-5k^2(r')^2-4rr'kk'-3r^2\right)^2+4r^2\left(kr'-4rk'\right)^2 = 0.
\end{equation}
As this is a sum of squares, both terms in (\ref{squares}) must be zero. The second term gives $r = ck^4$ for some constant $c$. Using the formula $r = -2k^{-1}k''$ for the Ricci scalar we obtain $k'' = -\frac 12ck^5$. The first term in (\ref{squares}) now imposes the first order equation
\begin{equation} \label{ode1}
    \frac{1}{16}+\frac 16ck^6 + (k')^2 = 0.
\end{equation}
We express $r$ in terms of $k$ and use (\ref{ode1}) to eliminate all derivatives of $k$ in res$(P_1,P_2)$. The resultant becomes a polynomial of degree 14 in $X_2$ with non-zero leading coefficient. We conclude that $\partial_x X_2 = 0$. But now we can view $P_1$ as a polynomial of degree 6 in $X_1$ with coefficients depending on $y$ only. The leading coefficient is non-zero, showing that $\partial_x X_1 = 0$ and therefore $\mathcal{L}_K X^\flat = 0$.

For the case $m = -4$ many terms vanish and the resultant of $P_1$ and $P_2$ is a polynomial of degree only 12 in $X_2$. Equating the leading term to zero leads to
\begin{equation*}
    (3r'kp-2p'kr)^2+p^2r^2=0.
\end{equation*}
It follows that we must have $p = 0$. After imposing this condition, res$(P_1,P_2)$ has degree 8 with non-zero leading coefficient. We again find $\partial_x X_2 = 0$, and as the leading coefficient of $X_1^6$ in $P_1$ is non-zero also $\partial_x X_1 = 0$ and therefore $\mathcal{L}_K X^\flat = 0$.

For the case $m = -1$ we return to (\ref{closed2}), (\ref{pqe1}) and (\ref{pqe2}), which now read
\begin{subequations}
\begin{align}
    \text{d}\Omega &= -3\Omega X^\flat +\star\text{d}R, \label{sm1}\\
    0 &= \Delta R + 3\langle X,\text{d}R\rangle+3\Omega^2, \label{sm2}\\
    0 &= -\text{d}\Delta R + 3\langle X,\text{d}R\rangle X^\flat -3\text{Hess}(R)(X)-\frac{15}{2}\Omega \star\!\text{d}R +\frac 32 R\:\text{d}R + 18\Omega^2X^\flat.\label{sm3}
\end{align}
\end{subequations}
Equation (\ref{sm2}) can be used to eliminate the $\langle X,\text{d}R\rangle$ term in (\ref{sm3}):
\begin{equation*}
   (3\text{Hess}(R) + (\Delta R-15\Omega^2)g)(X,\cdot) = -\text{d}\Delta R - \frac{15}{2}\Omega\star\text{d}R + \frac 32R\text{d}R.
\end{equation*}
We can solve this linear equation for $X$ in terms of $\Omega$ and the metric unless the determinant of the LHS is zero, which can only happen if $\partial_x\omega = 0$. If a solution for $X$ exists, plugging it into (\ref{sm2}) gives a polynomial equation for $\omega$ with non-zero leading coefficient. Hence in either case we must have $\partial_x \omega = 0$. Finally, (\ref{sm1}) is a linear equation for $X^\flat$ (unless $\Omega = 0$, but then the curvature is constant) with coefficients depending on $y$ only, which yields $\mathcal{L}_K X^\flat = 0$.
\end{proof}

\section{Hypergeometric differential equation} \label{Ab}

To solve (\ref{ode2C}) we first consider the third order linear ODE (\ref{Dalpha}) implied by it. Suppose $\alpha \neq 0$ and write $\nu = \alpha^{-1}$ and $z = \nu x =  \alpha^{-1}x$. The ODE (\ref{Dalpha}) reads
\begin{equation} \label{hypergeode}
    2z^3(z-1)C''' +z^2(8z-7)C''+z(4z-3+2\nu)C' + C\nu = 0,
\end{equation}
where the prime denotes a $z$-derivative. Equation (\ref{hypergeode}) has regular singular points at $0, 1$ and $\infty$ and can be solved in terms of the hypergeometric function ${}_3F_2$. Recall this function is defined as 
\begin{equation} \label{hypergeo}
{}_3F_2(a_1,a_2,a_3; b_1,b_2; z) = \sum_{n = 0}^\infty \frac{(a_1)_n(a_2)_n(a_3)_n}{(b_1)_n(b_2)_n}\frac{z^n}{n!},
\end{equation}
where the $a_i,b_i$ are complex parameters and $(a)_n = a(a+1)\cdots(a+n-1)$ with $(a)_0 = 1$. The power series (\ref{hypergeo}) converges for all complex $\vert z \vert < 1$. The solution to (\ref{hypergeode}) for $\vert z \vert < 1$ can be written as 
\begin{equation} \label{licomb0}
    C(z) = \beta C_1(z) + \gamma C_2(z) + \delta C_3(z),
\end{equation}
where $\beta,\gamma,\delta$ are (in general complex) constants and the $C_i$ denote the hypergeometric functions 
\begin{align*}
C_1(z) &= z^{-\frac 12}\:{}_3F_2\left(-\frac 12,-\frac 12,\frac 12\:;\:\frac 12-\sqrt{\nu},\frac 12+\sqrt{\nu}\:;\: z\right), \\
C_2(z) &= z^{\sqrt{\nu}}\:{}_3F_2\left(\sqrt{\nu},\sqrt{\nu},1 + \sqrt{\nu}\:;\:1 +2\sqrt{\nu}, \frac 32+\sqrt{\nu}\:;\:z\right), \\
C_3(z) &= z^{- \sqrt{\nu}}\:{}_3F_2\left(-\sqrt{\nu},-\sqrt{\nu},1 - \sqrt{\nu}\:;\:1 -2\sqrt{\nu}, \frac 32-\sqrt{\nu}\:;\:z\right).
\end{align*}
The solutions $C_1$ and $C_3$ are not valid when $\nu^{\frac 12}$ is a half-integer, as we would be dividing by zero in~(\ref{hypergeo}). In this case we replace $C_1$ and $C_3$ by Meijer G-functions as explained in \cite{CR08}. For $\alpha = 0$ the ODE (\ref{Dalpha}) has a singular point at $x = 0$ and can again be solved in terms of the Meijer G-function. The functions $C_i(z)$ are well-defined for all other values of $\nu$ provided we set $\sqrt{\nu} = i\sqrt{-\nu}$ for $\nu < 0$. Equation (\ref{hypergeode}) implies $L_\alpha C = $ const with $L_\alpha$ as in (\ref{ode2C}). Requiring this constant to be zero imposes a condition on the coefficients in (\ref{licomb0}). Using the series expansion (\ref{hypergeo}), this condition is found to be
\begin{equation} \label{coefcon}
    -\frac 14\beta^2 + 4\nu(1-4\nu)\gamma\delta = 0.
\end{equation}
Imposing this restriction gives the general solution to (\ref{ode2C}) around $z = 0$.  To obtain a Riemannian metric we require $C$ and $z$ to be real, with $z > 1$ when $\nu > 0$ and $z < 0$ when $\nu < 0$. The corresponding solution can be obtained by analytic continuation, where the constants $\beta,\gamma,\delta$ should be chosen such that the resulting function $C$ is real. For $z \in (0,1)$, for which the metric (\ref{gexp}) is Lorentzian, this condition is $\beta \in \R$ and $\gamma = \bar{\delta}$.

Let us now focus on the case where $\nu^{\frac 12}$ is an integer $n$. In this case the hypergeometric series in $C_3$ terminates and $C_3(z)$ equals $z^{- n}$ times a polynomial in $z$ of degree $n-1$. To simplify the expressions for $C_1$ and $C_2$ we may use the identities
\begin{align*}
    {}_3F_2\left(-\frac 12,-\frac 12,\frac 12; -\frac 12,\frac 12;z\right) =\sqrt{1-z}, \\
    {}_3F_2\left(1,1,1;2,\frac 32; z\right) = \frac{\arcsin(\sqrt{z})^2}{z},
\end{align*}
together with the differentiation formulae
\begin{align*}
    \left(z\tfrac{\text{d}}{\text{d}z}+ a_1\right){}_3F_2(a_1,a_2,a_3;b_1,b_2;z) &= a_1\:{}_3F_2(a_1+1,a_2,a_3;b_1,b_2;z), \\
     \left(z\tfrac{\text{d}}{\text{d}z} + b_1-1\right){}_3F_2(a_1,a_2,a_3;b_1,b_2;z) &= (b_1-1)\:{}_3F_2(a_1,a_2,a_3;b_1-1,b_2;z).
\end{align*}
For example, for $\nu = 1$ we obtain
\begin{equation*}
    C_1(z) = \tfrac{\arcsin(\sqrt{z}) + \sqrt{z(1-z)}}{2z}, \hspace{.5cm} C_2(z) = \tfrac{9z-3\arcsin(\sqrt{z})\left(2\sqrt{z(1-z)}+\arcsin(\sqrt{z})\right)}{z},\hspace{.5cm} C_3(z) = \tfrac{1}{z}.
\end{equation*}
For $z > 1$ we set $\sqrt{1-z} = i\sqrt{z-1}$ and arcsin$(\sqrt{z}) = \frac{\pi}{2}-i\arcsinh(\sqrt{z-1})$. We recover (\ref{Csol}) with
\begin{equation*}
    a = \delta + \frac{\pi}{4}\beta - \frac{3}{4}\pi^2\gamma, \hspace{.8cm} b = \frac i2\beta - 3\pi\gamma i, \hspace{.8cm} c = 3\gamma.
\end{equation*}
Note that (\ref{coefcon}) reduces to $b^2 = 4ac$. Using the equations above we can compute the hypergeometric functions $C_i$ for $\nu^{\frac 12} = 2,3,\dots$ and find more explicit solutions to (\ref{ode2C}).


\begin{thebibliography}{99}



\bibitem{ACS17}
D. Angella, S. Calamai and C. Spotti. On the Chern–Yamabe problem. Math. Res. Lett. \textbf{24}(3) (2017) 645-677.
\bibitem{BGKW22} E. Bahuaud, S. Gunasekaran, H. K. Kunduri and E. Woolgar. Static near-horizon geometries and rigidity of quasi-Einstein manifolds. Lett. Math. Phys. \textbf{112}, 6 (2022) 116.
\bibitem{BGKW23} E. Bahuaud, S. Gunasekaran, H. K. Kunduri and E. Woolgar. Deformations of the Kerr-(A)dS near horizon geometry. Class. Quant. Grav. \textbf{41} (2024) 065001.
\bibitem{B98} C. B\"ohm. Inhomogeneous Einstein metrics on low-dimensional spheres and other low-dimensional
 spaces. Invent. Math. \textbf{134}(1) (1998) 145-176.
 \bibitem{BDE} R. L. Bryant, M. Dunajski, and M. G. Eastwood. Metrisability of two-dimensional projective structures. J. Diff. Geom. {\bf 83} (2009) 465--499.
 \bibitem{C43} E. Cartan. Sur une classe d’espaces de Weyl. \'Ecole Norm. Sup. \textbf{60} (1943) 1-16.
\bibitem{CSW11} J. Case, Y-J. Shu and G. Wei. Rigidity of quasi-Einstein metrics. Diff. Geom. Appl. \textbf{29} (2011) 93–100.
\bibitem{C12} J. Case. Smooth metric measure spaces and quasi-Einstein metrics. Int. J. Math. \textbf{23} (2012) 1250110.
\bibitem{Hugo} E. Chac\'on and H. Garc\'ia-Compe\'an.
Self-dual gravity via Hitchin’s equations
J. Math. Phys.  {\bf 60} (2019) 052502.
\bibitem{CR08} E. S. Cheb-Terrab and A. D. Roche. Hypergeometric solutions for third order linear ODEs. {\tt arXiv:0803.3474} (2008).
\bibitem{CST18} P. T. Chru\'sciel, S. J. Szybka and K. P. Tod. Towards a classification of vacuum near-horizons geometries. Class. Quant. Grav. \textbf{35} (2018) 015002.
\bibitem{C24} E. Cochran. Killing vector fields on compact m-Quasi-Einstein manifolds. Proc. Amer. Math. Soc. \textbf{153} (2025) 841-849.
\bibitem{CDKL24} A. Colling, M. Dunajski, H. K. Kunduri and J. Lucietti. New quasi-Einstein metrics on a two-sphere. {\tt arXiv:2403.04117} (2024).
To appear in the Journal of Geometric Analysis.

\bibitem{CKL24} A. Colling, D. Katona and J. Lucietti. Rigidity of the extremal Kerr-Newman horizon. Lett. Math. Phys. \textbf{115}, 19 (2025). 

\bibitem{derdzinski} A. Derdzinski.  Connections with skew-symmetric Ricci tensor on surfaces. 
Result. Math. {\bf 52} (2008) 223–245.
\bibitem{DRSL}
D.~Dobkowski-Ry\l{}ko, W.~Kami\'nski, J.~Lewandowski and A.~Szereszewski.
The Near Horizon Geometry equation on compact 2-manifolds including the general solution for g \ensuremath{>} 0,
Phys. Lett. B \textbf{785} (2018) 381-385.

\bibitem{DMW} M. Dunajski, L. J. Mason, N M. J. Woodhouse.  From 2D Integrable Systems to Self-Dual gravity', J. Phys. A: Math. Gen  {\bf 31} (1998) 6019.

\bibitem{DSG} M. Dunajski.  Nonlinear graviton from the sine-Gordon equation. Twistor
Newletter  {\bf 40} 43-5; In {\em Mason, L.J.  et al.  Further advances in twistor
theory, \bf{3}} 85-87.

\bibitem{D24} M. Dunajski. Solitons, Instantons, and Twistors. 2nd edn. Oxford Graduate Texts in Mathematics
 \textbf{31}. Oxford University Press, Oxford (2024).
\bibitem{DL23} M. Dunajski and J. Lucietti. Intrinsic rigidity of extremal horizons. J. Diff. Geom. \textbf{132} (2026) 179--201.
\bibitem{G95} P. Gauduchon. Structures de Weyl-Einstein, espaces de twisteurs et vari\'et\'es de type $S^1 \times S^3$. J. reine angew. Math. \textbf{469} (1995) 1–50.

\bibitem{GH}  G. W. Gibbons, and  S. W. Hawking. Gravitational multi-instantons,
Phys. Lett. {\bf 78B} (1978) 430-432.

\bibitem{H87} N. Hitchin. The Self-Duality Equations on a Riemann surface. Proc. Lon. Math. Soc. \textbf{55} (1987) 59-126.
\bibitem{Hitchin2} N. J. Hitchin. Higgs bundles and diffeomorphism groups. In H.-D. Cao and
S.-T. Yau, editors, {\em Advances in geometry and mathematical physics. Surveys in differential geometry}, pages 139–163. International Press (2016) 139--163.
\bibitem{J09} J. Jezierski. On the existence of Kundt's metrics and degenerate (or extremal) Killing horizons. Class. Quant. Grav. \textbf{26} (2009) 035011. 
\bibitem{KL24} W. Kami\'nski and J. Lewandowski. Extreme horizon equation.
{\tt arXiv:2406.20068} (2024).
\bibitem{KW74} J. L. Kazdan and F. W. Warner. Curvature functions for compact 2-manifolds. Ann. Math. \textbf{99} (1974) 14-47.
\bibitem{K93} J. L. Kazdan. Applications of Partial differential Equations to Problems in Geometry (1993).
\bibitem{KK03} D.-S. Kim and Y. H. Kim. Compact Einstein warped product spaces with nonpositive scalar curvature. Proc. Amer. Math. Soc. \textbf{131} (2003) 2573-2576.
\bibitem{krynski} W. Kry\'nski.
Webs and projective structures on a plane. 
Diff. Geom. Appl. {\bf 37} (2014) 133.
\bibitem{KLR07} H.~K.~Kunduri, J.~Lucietti and H.~S.~Reall. Near-horizon symmetries of extremal black holes. Class. Quant. Grav. \textbf{24} (2007) 4169-4189.
\bibitem{KL13}
H.~K.~Kunduri and J.~Lucietti,
Classification of near-horizon geometries of extremal black holes.
Living Rev. Rel. \textbf{16} (2013) 8.
\bibitem{LP03} J. Lewandowski and T. Pawlowski. Extremal isolated horizons: a Local
 uniqueness theorem. Class. Quant. Grav. \textbf{20} (2003) 587-606.
\bibitem{L10} M. Limoncu. Modifications of the Ricci tensor and applications. Arch. Math. \textbf{95} (2010) 191-199.
\bibitem{LPP04} H. Lu, D. N. Page and C. N. Pope. New inhomogeneous Einstein metrics on sphere bundles over Einstein-K\"ahler manifolds. Phys. Lett. B \textbf{593} (2004) 218-226.
\bibitem{MW} L. J. Mason, and N. M. J. Woodhouse. Integrability, Self-duality
and Twistor theory, LMS Monograph, OUP. (1996)
\bibitem{M65} J. Milnor. Topology from the differentiable viewpoint.  Princeton Landmarks in Mathematics. Princeton University Press, Princeton, NJ (1997) reprint of the 1965 original.
\bibitem{NR16} P. Nurowski and M. Randall.
Generalized Ricci solitons. 
J. Geom. Anal. {\bf 26}  (2016) 1280–1345. 
\bibitem{PT93} H. Pedersen and K. P. Tod. Three-dimensional Einstein-Weyl geometry. Adv. Math. \textbf{97} (1993) 74–109.
\bibitem{Q97} Z. Qian. Estimates for weighted volumes and applications. Quart. J. Math. \textbf{48} (1997) 235-242.
\bibitem{R14} M. Randall.
{Local obstructions to projective surfaces admitting skew-symmetric Ricci tensor.}
Jour. Geom. Phys. {\bf 76} (2014) 192.
\bibitem{T92} K. P. Tod. Compact 3-dimensional Einstein-Weyl structures. J. London. Math. Soc. \textbf{45} (1992) 341.
\bibitem{WY16} W. Wylie and D. Yeroshkin. On the geometry of Riemannian manifolds with density. {\tt arXiv:1602.08000} (2016).
\bibitem{W23} W. Wylie. Rigidity of compact static near-horizon geometries with negative cosmological constant. Lett. Math. Phys. \textbf{113}, 29 (2023).
\bibitem{Y24} W. Yu. A note on Kazdan–Warner type equations on compact Riemannian manifolds. Non-linear Anal. \textbf{246} (2024) 113596.
\bibitem{Sokolov} A. V. Zhiber and V. V. Sokolov.
Exactly integrable hyperbolic equations of Liouville type.
Russian Math. Surveys {\bf 56} (2001)  61–101.


\end{thebibliography}
\end{document}